\documentclass{amsart}

\newtheorem{theorem}{Theorem}[section]
\newtheorem{lemma}[theorem]{Lemma}

\theoremstyle{definition}
\newtheorem{definition}[theorem]{Definition}

\newtheorem{proposition}[theorem]{Proposition}

\theoremstyle{remark}

\numberwithin{equation}{section}

\begin{document}

\title{Equations for the third secant variety of the Segre product of $n$ projective spaces}

\author{Yang Qi}
\address{Department of Mathematics, Texas A\&M University, College Station, Texas 77840}
\email{yangqi@math.tamu.edu}



\date{}



\begin{abstract}
We determine set theoretic defining equations for the third secant variety of the Segre product of $n$ projective spaces, and from the proof of the main statement we derive an upper bound for the degrees of these equations.
\end{abstract}

\maketitle

\section{Introduction}

The study of tensor ranks and border ranks plays an important role in complexity theory, algebraic statistics, biology, signal processing, and many other areas (\cite{LL, C, CJ, H, LC, ML}). Due to the geometric interpretations of rank and border rank, it is important to find the equations for the secant varieties of Segre varieties since they produce tests for the border rank of a tensor.

\subsection{Strassen's equations}

For a projective variety $X\subset\mathbb{P}W$, the $r$-th secant variety $\sigma_r(X)$ is defined by

\begin{equation*}
\sigma_r(X)=\overline{\displaystyle\bigcup_{x_1,\dots,x_r\in X} <x_1,\dots,x_r>}\subset\mathbb{P}W
\end{equation*}
where $<x_1,\dots,x_r>\subset \mathbb{P}W$ denotes the linear span of the points $x_1,\dots,x_r$.

Let $A_1, \dots, A_n$ be finite dimensional complex vector spaces, define the Segre variety $Seg(\mathbb{P}A_1\times\cdots\times\mathbb{P}A_n)$ to be the image of the map

\[
Seg: \mathbb{P}A_1 \times \cdots \times \mathbb{P}A_n \rightarrow \mathbb{P}(A_1 \otimes \cdots \otimes A_n)
\]
given by
\[
([v_1], \cdots, [v_n]) \mapsto [v_1 \otimes \cdots \otimes v_n].
\]
For $T\in A_1\otimes\cdots\otimes A_n$, in geometric language, $T$ is said to have rank $one$ if $[T]\in Seg(\mathbb{P}A_1\times\cdots\times\mathbb{P}A_n)$, and $T$ has border rank $\leq r$ if $[T]\in \sigma_r(Seg(\mathbb{P}A_1\times\cdots\times\mathbb{P}A_n))$. 

\begin{definition}
Given $W=A_1\otimes\cdots\otimes A_n$, a flattening of $W$ is a decomposition $W=(A_{i_1}\otimes\cdots\otimes A_{i_q})\otimes(A_{j_1}\otimes\cdots\otimes A_{j_{n-q}})=:A_I\otimes A_J$, where $I\cup J=\{1,\dots,n\}$ is a partition.
\end{definition}

Since $Seg(\mathbb{P}A_1\times\cdots\times\mathbb{P}A_n)\subset Seg(\mathbb{P}A_I\times\mathbb{P}A_J)$, $\sigma_r(Seg(\mathbb{P}A_1\times\cdots\times\mathbb{P}A_n))\subset\sigma_r(Seg(\mathbb{P}A_I\times\mathbb{P}A_J))$, the $(r+1)\times(r+1)$ minors of flattenings, i.e. $\wedge^{r+1}A^*_I\otimes\wedge^{r+1}A^*_J$, give equations for $\sigma_r(Seg(\mathbb{P}A_1\times\cdots\times\mathbb{P}A_n))$. For $\sigma_3(Seg(\mathbb{P}A\times\mathbb{P}B\times\mathbb{P}C))$, V. Strassen \cite{S} discovered equations beyond $4\times 4$ minors of flattenings. Here we present a version of Strassen's equations due to G. Ottaviani.

Given $T\in A\otimes B\otimes C$, i.e. $T: B^*\rightarrow A\otimes C$, $Id_A\otimes T$ gives a linear map $A\otimes B^*\rightarrow A\otimes A\otimes C$, by composing $Id_A\otimes T$ with the projection $\pi: A \otimes A\rightarrow \wedge^2A$ we obtain a map $T^{\wedge}_{BA}: A\otimes B^*\rightarrow \wedge^2A\otimes C$.

\begin{theorem}[\cite{O}]
Let $T\in A\otimes B\otimes C$, and assume $3\leq\dim A\leq \dim B\leq \dim C$. If $[T]\in\sigma_r(Seg(\mathbb{P}A\times\mathbb{P}B\times\mathbb{P}C))$, then $\operatorname{rank}(T^{\wedge}_{BA})\leq r(\dim A-1)$. Thus $(r(\dim A-1)+1)\times(r(\dim A-1)+1)$ minors of $T^{\wedge}_{BA}$ furnish equations for $\sigma_r(Seg(\mathbb{P}A\times\mathbb{P}B\times\mathbb{P}C))$, which are called Strassen's equations.
\end{theorem}


We have

\begin{theorem}[\cite{LM, F}]
\label{thm:str}
$\sigma_3(Seg(\mathbb{P}A\times\mathbb{P}B\times\mathbb{P}C))$ is the zero set of the size $4$ minors of flattenings and Strassen's equations.
\end{theorem}

When $\dim A=\dim B=\dim C=3$, we have $4\times 4$ minors of flattenings and degree $7$ Strassen's equations which define a zero set equivalent to Strassen's original degree $4$ equations \cite{LL}.

\subsection{Main results}

Given $T\in A_1\otimes\cdots\otimes A_n$, and any partition $I\cup J\cup K=\{1,\dots,n\}$, where $I, J, K$ disjoint, if $T\in\sigma_3(Seg(\mathbb{P}A_1\times\cdots\times\mathbb{P}A_n))$, then $T\in\sigma_3(Seg(\mathbb{P}A_I\times\mathbb{P}A_J\times\mathbb{P}A_K))$, so $T$ satisfies Strassen's equations for the partition $I\cup J\cup K=\{1,\dots,n\}$ and $4\times 4$ minors of all flattenings. Our main result is:

\begin{theorem}
The third secant variety of the Segre product of $n$ projective spaces $\sigma_3(Seg(\mathbb{P}A_1\times\cdots\times\mathbb{P}A_n))$ is set theoretically defined by Strassen's equations for all partitions $I\cup J\cup K=\{1,\dots,n\}$ and all $4\times 4$ minors of flattenings. More precisely, when $\dim A_i \leq 3$ for each $i$, $\sigma_3(Seg(\mathbb{P}A_1\times\cdots\times\mathbb{P}A_n))$ is set theoretically defined by Strassen's equations of degree $4$ for the partitions $\{i\}\cup \{j\}\cup \{1,\dots,\widehat{i},\cdots,\widehat{j},\cdots,n\}$ and all $4\times 4$ minors of flattenings.
\end{theorem}

As an example of the result of Draisma and Kuttler \cite{DK}, this theorem gives an explicit uniform upper bound of degrees of set theoretic defining equations for $\sigma_3 (Seg(\mathbb{P}A_1\times\cdots\times\mathbb{P}A_n))$.

\subsection{Normal forms of points in $\sigma_3(Seg(\mathbb{P}A_1\times\cdots\times\mathbb{P}A_n))$}

Here we recall results of Buczynski and Landsberg that classify all normal forms for tensors of border rank $\leq 3$.

\begin{proposition}[\cite{BL}]
\label{prop:norm}
Let $X$ denote $Seg(\mathbb{P}A_1\times\cdots\times\mathbb{P}A_n)$ and $p=[v]\in\sigma_2(X)$, then $v$ has one of the following normal forms:

$1$, $p\in X$;

$2$, $v=x+y$ with $[x], [y]\in X$;

$3$, $v=x'$ with $x'\in\widehat{T}_{[x]}X$.
\end{proposition}

\begin{theorem}[\cite{BL}]
\label{thm:normal0}
Let $X$ denote $Seg(\mathbb{P}A_1\times\cdots\times\mathbb{P}A_n)$ and $p=[v]\in\sigma_3(X)\setminus\sigma_2(X)$, then $v$ has one of the following normal forms:

$1$. $v=x+y+z$ with $[x], [y], [z]\in X$;

$2$. $v=x+x'+y$ with $[x], [y]\in X$ and $x'\in\widehat{T}_{[x]}X$;

$3$. $v=x+x'+x''$, where $[x(t)]\subset X$ is a curve and $x'=x'(0)$, $x''=x''(0)$;

$4$. $v=x'+y'$, where $[x], [y]\in X$ are distinct points that lie on a line contained in $X$, $x'\in\widehat{T}_{[x]}X$, and $y'\in\widehat{T}_{[y]}X$.
\end{theorem}

Normal forms for Theorem~\ref{thm:normal0} are as follows:

\begin{theorem}[\cite{BL}]
\label{thm:normal}
Let $X$ denote $Seg(\mathbb{P}A_1\times\cdots\times\mathbb{P}A_n)$ and $p=[v]\in\sigma_3(X)\setminus\sigma_2(X)$, then $v$ has one of the following normal forms:

$1$. $v = a^1_1\otimes \cdots \otimes a^n_1 + a^1_2\otimes \cdots \otimes a^n_2 + a^1_3\otimes \cdots \otimes a^n_3$;

$2$. $v = \displaystyle \sum_{i=1}^n a^1_1\otimes \cdots \otimes a^{i-1}_1 \otimes a^i_2 \otimes a^{i+1}_1 \otimes \cdots \otimes a^n_1 + a^1_3\otimes \cdots \otimes a^n_3$;

$3$. $v = \displaystyle \sum_{i<j} a^1_1\otimes \cdots \otimes a^{i-1}_1 \otimes a^i_2 \otimes a^{i+1}_1 \otimes \cdots \otimes a^{j-1}_1 \otimes a^j_2 \otimes a^{j+1}_1 \otimes \cdots \otimes a^n_1 + \sum_{i=1}^n a^1_1\otimes \cdots \otimes a^{i-1}_1 \otimes a^i_3 \otimes a^{i+1}_1 \otimes \cdots \otimes a^n_1$;

$4$. $v = \displaystyle \sum_{s=2}^n a^1_2\otimes a^2_1 \otimes \cdots \otimes a^{s-1}_1 \otimes a^s_2 \otimes a^{s+1}_1 \otimes \cdots \otimes a^n_1 + \sum_{i=1}^n a^1_1\otimes \cdots \otimes a^{i-1}_1 \otimes a^i_3 \otimes a^{i+1}_1 \otimes \cdots \otimes a^n_1$,

where $a^i_j\in A_i$, and the vectors need not all be linearly independent.
\end{theorem}

\subsection{Outline of the proof of the main result}

Given $T\in A_1\otimes\cdots\otimes A_n$, for each $A_i$ we fix a basis $\{a^i_j\}$ and its dual basis $\{\alpha^i_j\}$. Let $X_k:=Seg(\mathbb{P}A_1\times\cdots\times\mathbb{P}A_k\times\mathbb{P}(A_{k+1}\otimes\cdots\otimes A_n))$, and $X:=Seg(\mathbb{P}A_1\times\cdots\times\mathbb{P}A_n)$.

\begin{definition}[\cite{LW}]
Given $b_1,\dots,b_n\in \mathbb{Z}_{\geq 1}$, define the subspace variety 
\[
Sub_{b_1,\dots,b_n}(A_1\otimes \cdots\otimes A_n):= \mathbb{P}\{T\in A_1\otimes\cdots\otimes A_n|\dim T(A^*_j)\leq b_j, \forall 1\leq j\leq n\}.
\]
\end{definition}

\begin{proposition}[\cite{JLLM}]
\label{cor:inhe}
Let $\dim A_j\geq r$, $1\leq j\leq n$. The ideal of $\sigma_r(Seg(\mathbb{P}A_1\times \cdots \times \mathbb{P}A_n))$ is generated by the modules inherited from the ideal of $\sigma_r(Seg(\mathbb{P}^{r-1} \times \cdots \times \mathbb{P}^{r-1}))$ and the modules generating the ideal of $Sub_{r, \dots, r}(A_1\otimes \cdots \otimes A_n)$. The analogous scheme and set theoretic results hold as well.
\end{proposition}

According to this proposition, we only need to consider the case $3\geq\dim A_1\geq\cdots\geq\dim A_n\geq 2$.

\begin{proof}[\textbf{Outline of the proof of the main result}]

Here we present the main idea of the proof, and we will show the details in the next section.

If $T$ satisfies Strassen's equations of the partition $\{1\}\cup\{2\}\cup\{3,\dots,n\}$ and $4 \times 4$ minors of flattenings, then $T\in\sigma_3 (X_2)$. We split our discussion into $4$ cases to show $T\in \sigma_3 (X)$. Much of the proof is a careful case by case analysis. Case 1 and Case 4 are straightforward. The subtle case is Case 3, where we exploit knowledge about symmetric tensors and need the results of other cases.

\textbf{Case 1}: $T\in\sigma_3(X_2)\setminus\sigma_2(X_2)$ and $T\notin Sub_{3,2,\dots,2}(A_1\otimes \cdots\otimes A_n)$, then $T$ has one of the four types of the normal forms in Theorem~\ref{thm:normal} for $\sigma_3 (X_2)$. Because $4\times 4$ minors of $T: A^*_1\otimes A^*_3\rightarrow A_2\otimes A_4\otimes\cdots\otimes A_n$ vanish, $T$ has to have the same type of normal form for $\sigma_3 (X_3)$. Similarly, by considering $4\times 4$ minors of $T: A^*_1\otimes A^*_k\rightarrow A_2\otimes\cdots\otimes\widehat{A_k}\otimes\cdots\otimes A_n$ we use induction to show that $T$ has to maintain the same type of normal form for $\sigma_3 (X)$.

\textbf{Case 2}: $T\in\sigma_3(X_2)\setminus\sigma_2(X_2)$ and $T\in Sub_{3,2,\dots,2}(A_1\otimes \cdots\otimes A_n)\setminus Sub_{2,2,\dots,2}(A_1\otimes \cdots\otimes A_n)$, then $T$ has one of the normal forms in Theorem~\ref{thm:normal} for $\sigma_3 (X_2)$. Because $\dim A_2=\cdots=\dim A_n=2$, the discussion of this case is more complicated than \textbf{Case 1}, and we split the argument into several subcases for each type of normal forms. For each subcase, by considering $4\times 4$ minors of $T: A^*_1\otimes A^*_3\rightarrow A_2\otimes A_4\otimes\cdots\otimes A_n$ and $T: A^*_2\otimes A^*_3\rightarrow A_1\otimes A_4\otimes\cdots\otimes A_n$, we show $T$ has one of the normal forms for $\sigma_3 (X_3)$. Note that the type of the normal form of $T$ for $\sigma_3(X_2)$ could be different from the type of the normal form of $T$ for $\sigma_3(X_3)$. By induction, we show that $T$ has one of the normal forms of points in $\sigma_3 (X)$.

\textbf{Case 3}: $T\in\sigma_3(X_2)\setminus\sigma_2(X_2)$ and $T\in Sub_{2,2,\dots,2}(A_1\otimes \cdots\otimes A_n)$. In this case, $T$ has two types of normal forms, $T = (a^1_1 \otimes a^2_1 + a^1_2 \otimes a^2_2) \otimes b^3_1 + a^1_1 \otimes a^2_2 \otimes b^3_2 + a^1_2 \otimes a^2_1 \otimes b^3_3$ or $T = a^1_1 \otimes a^2_1 \otimes b^3_1 + a^1_1 \otimes a^2_2 \otimes b^3_2 + a^1_2 \otimes a^2_1 \otimes b^3_3$ for some $b^3_j \in A_3 \otimes \cdots \otimes A_n$. For the generic normal form $T = (a^1_1 \otimes a^2_1 + a^1_2 \otimes a^2_2) \otimes b^3_1 + a^1_1 \otimes a^2_2 \otimes b^3_2 + a^1_2 \otimes a^2_1 \otimes b^3_3$, we show that there is a rank $2$ matrix $\phi_{21}$ in the kernel of $T^{\wedge}_{A_2A_1}: A_1\otimes A^*_2\rightarrow A_3\otimes\cdots\otimes A_n$, and $\phi_{21}(T)\in S^2A_1\otimes (A_3\otimes\cdots\otimes A_n)$. So if for each $2\leq i\leq n$, $T$ has the generic type of normal form for $\sigma_3(Seg(\mathbb{P}A_1\times\mathbb{P}A_i\times\mathbb{P}(A_2\otimes\cdots\otimes\widehat{A_i}\otimes\cdots\otimes A_n)))$, then similarly we have a $2\times 2$ matrix $\phi_{i1}\in Ker(T^{\wedge}_{A_iA_1})$ with full rank, and $\phi_{n1}\circ\cdots\circ\phi_{21}(T)\in S^nA_1$. Since each $\phi_{i1}$ is nonsingular, $T\in \sigma_3 (X)$ if and only if $\phi_{n1}\circ\cdots\circ\phi_{21}(T)\in \sigma_3 (\nu_n(\mathbb{P}A_1))$, where $\nu_n$ is the $n$-th Veronese embedding. Since the equations for $\sigma_3(\nu_n(\mathbb{P}^1))$ are known \cite{LO}, we can check Strassen's equations and $4\times 4$ minors of flattenings give equations for $\sigma_3 (X)$ in this situation. If for some $2\leq i\leq n$, say $i=2$, $T$ does not have the generic normal form for $\sigma_3 (X_2)$, $T$ must have the other type of normal form $T = a^1_1 \otimes a^2_1 \otimes b^3_1 + a^1_1 \otimes a^2_2 \otimes b^3_2 + a^1_2 \otimes a^2_1 \otimes b^3_3$. By considering $4\times 4$ minors of $T: A^*_1\otimes A^*_3 \rightarrow A_2\otimes A_4\otimes\cdots\otimes A_n$, $T: A^*_2\otimes A^*_3 \rightarrow A_1\otimes A_4\otimes\cdots\otimes A_n$, and $T: A^*_1\otimes A^*_2\otimes A^*_3 \rightarrow A_4\otimes\cdots\otimes A_n$, we deduce $T\in \sigma_3 (X_3)$. Then we use induction to show $T\in \sigma_3 (X)$ by checking each type of the normal forms in Theorem~\ref{thm:normal}, under the assumption that $T$ is not of the generic normal form for $\sigma_3 (X_2)$. When proceeding by induction, because $\dim T(A^*_3\otimes\cdots\otimes A^*_n)\leq 3$ we can view $T$ as a tensor in $T(A^*_3\otimes\cdots\otimes A^*_n) \otimes A_3 \otimes \cdots \otimes A_n$ and reduce most cases to \textbf{Case 2}. For the remaining cases, we show directly $T\in \sigma_3 (X)$.

\textbf{Case 4}: $T\in\sigma_2(X_2)$, then $T$ has one of the three types of the normal forms in Proposition~\ref{prop:norm} for $\sigma_3 (X_2)$. We verify by induction that for each normal form $T\in\sigma_3 (X)$.

\end{proof}

\section{Proof of the main theorem}

We only need to show that if $T$ satisfies Strassen's equations of all partitions $I\cup J\cup K = \{1,\dots, n\}$ and $4 \times 4$ minors of all flattenings $I \cup J =\{1,\dots,n\}$, then $T\in \sigma_3(Seg(\mathbb{P}A_1 \times \cdots \times \mathbb{P}A_n))$. For each $A_i$ we fix a basis $\{a^i_j\}$ and its dual basis $\{\alpha^i_j\}$. Let $X_k:=Seg(\mathbb{P}A_1\times\cdots\times\mathbb{P}A_k\times\mathbb{P}(A_{k+1}\otimes\cdots\otimes A_n))$, and $X:=Seg(\mathbb{P}A_1\times\cdots\times\mathbb{P}A_n)$. For any flattening $I \cup J = \{1,\dots, n\}$, $4 \times 4$ minors of $T: A^*_I \rightarrow A_J$ vanish if and only if $\dim T(A^*_I) \leq 3$. By Corollary~\ref{cor:inhe}, we can assume $3 \geq \dim A_1 \geq \cdots \geq \dim A_n \geq 2$. Since $T$ satisfies Strassen's equations of the partition $\{1\}\cup\{2\}\cup\{3,\dots,n\}$ and $4\times 4$ minors of all flattenings, by Theorem~\ref{thm:str} we have $T\in\sigma_3 (X_2)$. We split our discussion into $4$ cases to show $T\in\sigma_3 (X)$. 

\subsection{Case 1: $T\in\sigma_3(X_2)\setminus\sigma_2(X_2)$, $T\notin Sub_{3,2,\dots,2}(A_1\otimes \cdots\otimes A_n)$}

Since $T$ has one of the normal forms in Theorem~\ref{thm:normal}, we use induction to show $T\in\sigma_3 (X)$ by verifying each normal form.

\textbf{Type 1}: Without loss of generality, let $T=a^1_1\otimes a^2_1\otimes u_1+a^1_2\otimes a^2_2\otimes u_2+a^1_3\otimes a^2_3\otimes u_3$, where $u_i\in A_3\otimes\cdots\otimes A_n$. $\dim T(A^*_1\otimes A^*_3)\leq 3$ implies that $u_i: A_3^*\rightarrow A_4\otimes \cdots\otimes A_n$ has rank $\leq 1$ for all $i$, say $u_i=b^3_i\otimes v_i$ for some $b^3_i\in A_3$ and $v_i\in A_4\otimes\cdots\otimes A_n$. Therefore $T=a^1_1\otimes a^2_1\otimes b^3_1\otimes v_1+a^1_2\otimes a^2_2\otimes b^3_2 \otimes v_2+a^1_3\otimes a^2_3\otimes b^3_3\otimes v_3$, i.e. $T \in \sigma_3(X_3)$.


Now we use induction, assume $T=a^1_1\otimes a^2_1\otimes b^3_1\otimes\cdots\otimes b^k_1+a^1_2\otimes a^2_2\otimes b^3_2 \otimes\cdots\otimes b^k_2+a^1_3\otimes a^2_3\otimes b^3_3\otimes\cdots\otimes b^k_3$, then $\dim T(A^*_1\otimes A^*_k)\leq 3$ implies that $b^k_i: A_k^*\rightarrow A_{k+1}\otimes \cdots\otimes A_n$ has rank $\leq 1$ for all $1 \leq i \leq 3$.


\textbf{Type 2}: $T=a^1_1\otimes a^2_1\otimes v^3_2+a^1_1\otimes a^2_2\otimes v^3_1+a^1_2\otimes a^2_1\otimes v^3_1+a^1_3\otimes a^2_3\otimes v^3_3$, where $v^3_i\in A_3\otimes\cdots\otimes A_n$. Since $T\notin \sigma_2(X_2)$, $v^3_1$ and $v^3_3$ are non-zero. $\dim T(A^*_1\otimes A^*_3)\leq 3$ implies $v^3_1$ and $v^3_3: A_3^*\rightarrow A_4\otimes\cdots\otimes A_n$ have rank $1$, say $v^3_i=b^3_i\otimes v^4_i$ for $i=1,3$ and some $b^3_i\in A_3$, $v^4_i\in A_4\otimes\cdots\otimes A_n$, and for each $j=2,3$, $a^2_1\otimes v^3_2(\alpha^3_j)+a^2_2\otimes v^3_1(\alpha^3_j)$ is a linear combination of $a^2_1\otimes v^3_2(\alpha^3_1)+a^2_2\otimes v^3_1(\alpha^3_1)$ and $a^2_1\otimes v^3_1(\alpha^3_1)$, then $v^3_2=b^3_1\otimes v^4_2+b^3_2\otimes v^4_1$ for some $b^3_2\in A_3$ and $v^4_2\in A_4\otimes\cdots\otimes A_n$. Thus $T=a^1_2\otimes a^2_1\otimes b^3_1\otimes v^4_1+a^1_1\otimes a^2_1\otimes b^3_1\otimes v^4_2+a^1_1\otimes a^2_1\otimes b^3_2\otimes v^4_1+a^1_1\otimes a^2_2\otimes b^3_1\otimes v^4_1+a^1_3\otimes a^2_3\otimes b^3_3\otimes v^4_3$.


Now we use induction, and assume that $T=\displaystyle\sum_{i=1}^k b^1_1\otimes\cdots\otimes b^{i-1}_1\otimes b^i_2\otimes b^{i+1}_1\otimes\cdots\otimes b^k_1+b^1_3\otimes\cdots\otimes b^k_3$, where $b^i_j=a^i_j$ for $i=1,2$ and $1\leq j\leq 3$. The induction argument is similar to the case $k=3$ above.


\textbf{Type 3}: $T=a^1_1\otimes a^2_2\otimes v^3_2+a^1_2\otimes a^2_1\otimes v^3_2+a^1_2\otimes a^2_2\otimes v^3_1+a^1_1\otimes a^2_1\otimes v^3_3+a^1_1\otimes a^2_3\otimes v^3_1+a^1_3\otimes a^2_1\otimes v^3_1$, where $v^3_i\in A_3\otimes\cdots\otimes A_n$. If $v^3_1=0$, $T$ has been discussed in \textbf{Case 1 Type 1}. If $v^3_2=0$, $T$ has been discussed in \textbf{Case 1 Type 2}. So we assume $v^3_1$ and $v^3_2$ are non-zero. $\dim T(A^*_1\otimes A^*_3)\leq 3$ implies $v^3_1=u^3_1\otimes u^4_1$, $v^3_2=u^3_1\otimes u^4_2+u^3_2\otimes u^4_1$ and $v^3_3=u^3_1\otimes u^4_3+u^3_2\otimes u^4_2+u^3_3\otimes u^4_1$ for some $u^3_1, u^3_2, u^3_3\in A_3$, and $u^4_1, u^4_2, u^4_3\in A_4\otimes\cdots\otimes A_n$. Denote $a^i_j$ by $u^i_j$ when $i=1,2$, then $T=\displaystyle\sum_{1\leq i<j\leq 4}u^1_1\otimes\cdots\otimes u^i_2\otimes\cdots\otimes u^j_2\otimes\cdots\otimes u^4_1+\displaystyle\sum_{i=1}^4 u^1_1\otimes\cdots\otimes u^i_3\otimes\cdots\otimes u^4_1$.


The induction argument is similar the above argument.


\textbf{Type 4}: $T=a^1_2\otimes a^2_1\otimes v^3_2+a^1_2\otimes a^2_2\otimes v^3_1+a^1_1\otimes a^2_1\otimes v^3_3+a^1_1\otimes a^2_3\otimes v^3_1+a^1_3\otimes a^2_1\otimes v^3_1$ for some $v^3_j\in A_3\otimes\cdots\otimes A_n$. Since $T\notin \sigma_2(X_2)$, $v^3_1\neq 0$, then $\dim T(A^*_1\otimes A^*_3)\leq 3$ implies $v^3_1=u^3_1\otimes u^4_1$, $v^3_2=u^3_1\otimes u^4_2+u^3_2\otimes u^4_1$, $v^3_3=u^3_1\otimes u^4_3+u^3_3\otimes u^4_1$ for some $u^3_j\in A_3$, $u^4_j\in A_4\otimes\cdots\otimes A_n$. Denote $a^i_j$ by $u^i_j$ for $i=1,2$, then $T=\displaystyle\sum_{i=2}^4 u^1_2\otimes\cdots\otimes u^i_2\otimes\cdots\otimes u^4_1+\displaystyle\sum_{i=1}^4 u^1_1\otimes\cdots\otimes u^i_3\otimes\cdots\otimes u^4_1$.


The induction argument is similar.

\subsection{Case 2: $T\in\sigma_3(X_2)\setminus\sigma_2(X_2)$, $T\in Sub_{3,2,\dots,2}(A_1\otimes \cdots\otimes A_n)\setminus Sub_{2,2,\dots,2}(A_1\otimes \cdots\otimes A_n)$}

We show $T\in\sigma_3 (X)$ by induction on each type of the normal forms.

\textbf{Type 1}: $T=a^1_1\otimes b^2_1\otimes b^3_1+a^1_2\otimes b^2_2\otimes b^3_2+a^1_3\otimes b^2_3\otimes b^3_3$, where $b^2_j\in A_2$ and $b^3_j\in A_3\otimes\cdots\otimes A_n$. Without loss of generality, we can assume $b^2_1$ and $b^2_2$ are linearly independent, then $b^2_3=b^2_1$ or $b^2_3=b^2_1+b^2_2$.

If $b^2_3=a^2_1$, since $\dim T(A^*_2\otimes A^*_3)\leq 3$, then either $b^3_2: A^*_3\rightarrow A_4\otimes\cdots\otimes A_n$ has rank $1$, or both $b^3_1$ and $b^3_3$ have rank $1$ as maps $A^*_3\rightarrow A_4\otimes\cdots\otimes A_n$.

When $b^3_2: A^*_3\rightarrow A_4\otimes\cdots\otimes A_n$ has rank $1$, let $b^3_2=a^3_2\otimes b^4_2$ for some $b^4_2 \in A_4 \otimes \cdots \otimes A_n$. We only need to consider the case that at least one of $b^3_1$ and $b^3_3: A^*_3 \rightarrow A_4 \otimes \cdots \otimes A_n$ has rank $2$. Without loss of generality we can assume $b^3_1=u^3_1\otimes b^4_1 + u^3_3\otimes b^4_3$ for some $u^3_i \in A_3$ and $b^4_i \in A_4 \otimes \cdots \otimes A_n$ where $i = 1, 3$, then $\dim T(A^*_1\otimes A^*_3)\leq 3$ requires $b^3_3(\alpha^3_j)=x_jb^4_1+y_jb^4_3$ for some $x_j, y_j$, where $j=1,2$. Consider $A_3\otimes V_4$, where $V_4$ is spanned by $b^4_1$ and $b^4_3$, after a change of basis, we can assume $b^3_1=u^3_1\otimes b^4_1+u^3_3\otimes b^4_3$ and $b^3_3=\lambda u^3_1\otimes b^4_1+u^3_1\otimes b^4_3+\lambda u^3_3\otimes b^4_3$, or $b^3_3=\mu u^3_1\otimes b^4_1+\nu u^3_3\otimes b^4_3$. Then $T=T'+a^1_2\otimes b^2_2\otimes a^3_2\otimes b^4_2$, where $T'=(a^1_1+\lambda a^1_3)\otimes b^2_1\otimes u^3_1\otimes b^4_1+(a^1_1+\lambda a^1_3)\otimes b^2_1\otimes u^3_3\otimes b^4_3+a^1_3\otimes b^2_1\otimes u^3_1\otimes b^4_3 \in \widehat{T}_{(a^1_1+\lambda a^1_3)\otimes b^2_1\otimes u^3_1\otimes b^4_3} X_3$, or $T=(a^1_1+\mu a^1_3)\otimes b^2_1 \otimes u^3_1\otimes b^4_1+(a^1_1+\nu a^1_3)\otimes b^2_1\otimes u^3_3\otimes b^4_3+a^1_2\otimes b^2_2\otimes a^3_2\otimes b^4_2$.

When $b^3_1$ and $b^3_3: A^*_3\rightarrow A_4\otimes\cdots\otimes A_n$ have rank $1$, say $b^3_1=a^3_1\otimes b^4_1$ and $b^3_3=u^3_3\otimes b^4_3$ for some $u^3_3 \in A_3$ and $b^4_i \in A_4 \otimes \cdots \otimes A_n$ where $i = 1, 3$, and assume $b^3_2: A^*_3 \rightarrow A_4 \otimes \cdots \otimes A_n$ has rank $2$, $\dim T(A^*_2\otimes A^*_3)\leq 3$ requires $u^3_3=a^3_1$ up to a scalar, and $\dim T(A^*_1\otimes A^*_3)\leq 3$ requires $b^4_1=b^4_3$ up to a scalar, then $T=(a^1_1+a^1_3)\otimes b^2_1\otimes a^3_1\otimes b^4_1 + a^1_2\otimes b^2_2\otimes a^3_1\otimes b^3_2(\alpha^3_1) + a^1_2\otimes b^2_2\otimes a^3_2 \otimes b^3_2(\alpha^3_2)$.

If $b^2_3=b^2_1+b^2_2$, $\dim T(A^*_2\otimes A^*_3)\leq 3$ implies $b^3_1$ or $b^3_2: A^*_3 \rightarrow A_4 \otimes \cdots \otimes A_n$ has rank $1$. If only one of them has rank $1$, without loss of generality we assume that $b^3_2=a^3_1\otimes u^4_1+a^3_2\otimes u^4_2$, and $b^3_1=u^3_1\otimes u^4_3$. $\dim T(A^*_2\otimes A^*_3)\leq 3$ implies $b^3_3=u^3_1\otimes u^4_4$ for some $u^4_4\in A_4\otimes\cdots\otimes A_n$. $\dim T(A^*_1\otimes A^*_3)\leq 3$ requires that $u^4_3$ and $u^4_4$ are linearly dependent, then we can assume $u^4_4=u^4_3$. $\dim T(A^*_1\otimes A^*_3)\leq 3$ also requires $u^4_4$ is a linear combination of $u^4_1$ and $u^4_2$. Consider $A_3\otimes V_4$, where $V_4$ is the subspace of $A_4\otimes\cdots\otimes A_n$ spanned by $u^4_1$ and $u^4_2$, after a change of basis, we can assume $b^3_2=a^3_1\otimes u^4_1+a^3_2\otimes u^4_2$ is still the identity matrix, and $b^3_1=b^3_3=a^3_1\otimes u^4_2$ or $a^3_1\otimes u^4_1$. Then $T=(a^1_1+a^1_3)\otimes b^2_1\otimes a^3_1\otimes u^4_2+T'$, where $T'=a^1_2\otimes b^2_2\otimes a^3_1\otimes u^4_1+a^1_2\otimes b^2_2\otimes a^3_2\otimes u^4_2+a^1_3\otimes b^2_2\otimes a^3_1\otimes u^4_2\in \widehat{T}_{a^1_2\otimes b^2_2\otimes a^3_1\otimes u^4_2} X_3$, or $T=(a^1_1+a^1_3)\otimes b^2_1\otimes a^3_1\otimes u^4_1+(a^1_2+a^1_3)\otimes b^2_2\otimes a^3_1\otimes u^4_1+a^1_2\otimes b^2_2\otimes a^3_2\otimes u^4_2$.

If both $b^3_1$ and $b^3_2$ have rank $1$, let $b^3_1=a^3_1\otimes u^4_1$ and $b^3_2=u^3_2\otimes u^4_2$. If $u^4_1$ and $u^4_2$ are linearly independent, $\dim T(A^*_1\otimes A^*_3)\leq 3$ implies $b^3_3: A^*_3 \rightarrow A_4 \otimes \cdots \otimes A_n$ has rank $1$. If $u^4_1$ and $u^4_2$ are dependent, say $u^4_1=u^4_2$, and if $u^3_2=a^3_1$ up to a scalar, since $\dim T(A^*_1\otimes A^*_3)\leq 3$, then $b^3_3(\alpha^3_1)=xb^3_3(\alpha^3_2)+yu^4_1$ for some $x$, $y$. So $T=(a^1_1+ya^1_3)\otimes b^2_1\otimes a^3_1\otimes u^4_1+(a^1_2+ya^1_3)\otimes b^2_2\otimes a^3_1\otimes u^4_1+a^1_3\otimes(b^2_1+b^2_2)\otimes (xa^3_1+a^3_2)\otimes b^3_3(\alpha^3_2)$. If $u^3_2$ and $a^3_1$ are independent, we can assume $u^3_2=a^3_2$, since $\dim T(A^*_2\otimes A^*_3)\leq 3$, then $b^3_3: A^*_3 \rightarrow A_4 \otimes \cdots \otimes A_n$ has rank $1$.


Now we use induction. Assume $T=a^1_1\otimes b^2_1\otimes\cdots\otimes b^k_1+a^1_2\otimes b^2_2\otimes\cdots\otimes b^k_2+a^1_3\otimes b^2_3\otimes\cdots\otimes b^k_3$, without loss of generality we can assume $b^2_1=a^2_1$, $b^2_2=a^2_2$, then $b^2_3=a^2_1$ or $b^2_3=a^2_1+a^2_2$. The induction argument is similar to the case $k=3$.


\textbf{Type 2}: $T=a^1_1\otimes b^2_1\otimes b^3_2+a^1_1\otimes b^2_2\otimes b^3_1+a^1_2\otimes b^2_1\otimes b^3_1+a^1_3\otimes b^2_3\otimes b^3_3$, without loss of generality we can assume $b^2_1=a^2_1$ and $b^2_2=a^2_2$, then $b^2_3=a^2_1$, or $b^2_3=a^2_2+\lambda a^2_1$ for some $\lambda\in\mathbb{C}$.

When $b^2_3=a^2_1$, $\dim T(A^*_2\otimes A^*_3)\leq 3$ forces $b^3_1: A^*_3 \rightarrow A_4 \otimes \cdots \otimes A_n$ has rank $1$, say $b^3_1=a^3_1\otimes b^4_1$. If $b^3_3: A^*_3\rightarrow A_4\otimes\cdots\otimes A_n$ has rank $2$, say $b^3_3=a^3_1\otimes b^4_2+a^3_2\otimes b^4_3$, then $\dim T(A^*_1\otimes A^*_3)\leq 3$ requires that $b^4_1$ and $b^3_2(\alpha^3_2)$ are both in the subspace spanned by $b^4_2$ and $b^4_3$. After a change of basis, we can assume that $b^3_3=a^3_1\otimes b^4_2+a^3_2\otimes b^4_3$, and $b^3_1=a^3_1\otimes b^4_2$ or $b^3_1=a^3_1\otimes b^4_3$. We can assume $b^3_2(\alpha^3_2)=b^4_2+\lambda b^4_3$ or $b^3_2(\alpha^3_2)=b^4_3$. So we have four cases:

Case 1: If $b^3_1=a^3_1\otimes b^4_3$ and $b^3_2(\alpha^3_2)=b^4_2+\lambda b^4_3$, $T=a^1_1\otimes a^2_1\otimes a^3_1\otimes b^3_2(\alpha^3_1)+a^1_1\otimes a^2_1\otimes (\lambda a^3_2)\otimes b^4_3+a^1_1\otimes a^2_2\otimes a^3_1\otimes b^4_3+a^1_2\otimes a^2_1\otimes a^3_1\otimes b^4_3+a^1_1\otimes a^2_1\otimes a^3_2\otimes b^4_2+a^1_3\otimes a^2_1\otimes a^3_1\otimes b^4_2+a^1_3\otimes a^2_1\otimes a^3_2\otimes b^4_3$. Let $S(t)=(a^1_1+ta^1_3+t^2a^1_2)\otimes(a^2_1+t^2a^2_2)\otimes(a^3_1+ta^3_2+t^2\lambda a^3_2)\otimes(b^4_3+tb^4_2+t^2b^3_2(\alpha^3_1))$, then $T=S''(0)$.

Case 2: If $b^3_1=a^3_1\otimes b^4_3$ and $b^3_2(\alpha^3_2)=b^4_3$, then $T=T'+T''$, where $T'=a^1_1\otimes a^2_1\otimes a^3_1\otimes b^3_2(\alpha^3_1)+a^1_1\otimes a^2_1\otimes a^3_2\otimes b^4_3+a^1_1\otimes a^2_2\otimes a^3_1\otimes b^4_3+a^1_2\otimes a^2_1\otimes a^3_1\otimes b^4_3\in\widehat{T}_{a^1_1\otimes a^2_1\otimes a^3_1\otimes b^4_3} X_3$, and $T''=a^1_3\otimes a^2_1\otimes a^3_1\otimes b^4_2+a^1_3\otimes a^2_1\otimes a^3_2\otimes b^4_3 \in \widehat{T}_{a^1_3\otimes a^2_1\otimes a^3_1\otimes b^4_3} X_3$.

Case 3: If $b^3_1=a^3_1\otimes b^4_2$ and $b^3_2(\alpha^3_2)=b^4_2+\lambda b^4_3$, $T=T'+(\lambda a^1_1+a^1_3)\otimes a^2_1\otimes a^3_2\otimes b^4_3$, where $T'=a^1_1\otimes a^2_1\otimes a^3_1\otimes b^3_2(\alpha^3_1)+a^1_1\otimes a^2_1\otimes a^3_2\otimes b^4_2+a^1_1\otimes a^2_2\otimes a^3_1\otimes b^4_2+(a^1_2+a^1_3)\otimes a^2_1\otimes a^3_1\otimes b^4_2\in \widehat{T}_{a^1_1\otimes a^2_1\otimes a^3_1\otimes b^4_2} X_3$.

Case 4: If $b^3_1=a^3_1\otimes b^4_2$ and $b^3_2(\alpha^3_2)=b^4_3$, then $T=T'+(a^1_1+a^1_3)\otimes a^2_1\otimes a^3_2\otimes b^4_3$, where $T'=a^1_1\otimes a^2_1\otimes a^3_1\otimes b^3_2(\alpha^3_1)+a^1_1\otimes a^2_2\otimes a^3_1\otimes b^4_2+(a^1_2+a^1_3)\otimes a^2_1\otimes a^3_1\otimes b^4_2\in\widehat{T}_{a^1_1\otimes a^2_1\otimes a^3_1\otimes b^4_2} X_3$.

If $b^3_3: A^*_3 \rightarrow A_4 \otimes \cdots \otimes A_n$ has rank $1$, say $b^3_3=(xa^3_1+ya^3_2)\otimes b^4_3$, and $b^4_1$ and $b^4_3$ are linearly independent, $\dim T(A^*_1\otimes A^*_3)\leq 3$ forces $b^3_2(\alpha^3_2)$ is a linear combination of $b^4_1$ and $b^4_3$. We can assume $b^3_2(\alpha^3_2)=b^4_1$ or $b^3_2(\alpha^3_2)=b^4_3+\lambda b^4_1$. If $b^3_2(\alpha^3_2)=b^4_1$, $T=T'+a^1_3\otimes a^2_1\otimes (xa^3_1+ya^3_2)\otimes b^4_3$, where $T'=a^1_1\otimes a^2_1\otimes a^3_1\otimes b^3_2(\alpha^3_1)+a^1_1\otimes a^2_1\otimes a^3_2\otimes b^4_1+a^1_1\otimes a^2_2\otimes a^3_1\otimes b^4_1+a^1_2\otimes a^2_1\otimes a^3_1\otimes b^4_1\in\widehat{T}_{a^1_1\otimes a^2_1\otimes a^3_1\otimes b^4_1} X_3$. If $b^3_2(\alpha^3_2)=b^4_3+\lambda b^4_1$, we can assume $b^3_3=a^3_2\otimes b^4_3$ or $b^3_3=(a^3_1+\mu a^3_2)\otimes b^4_3$. If $b^3_3=a^3_2\otimes b^4_3$, then $T=T'+(a^1_1+a^1_3)\otimes a^2_1\otimes a^3_2\otimes b^4_3$, where $T'=a^1_1\otimes a^2_1\otimes a^3_1\otimes b^3_2(\alpha^3_1)+a^1_1\otimes a^2_1\otimes(\lambda a^3_2)\otimes b^4_1+a^1_1\otimes a^2_2\otimes a^3_1\otimes b^4_1+a^1_2\otimes a^2_1\otimes a^3_1\otimes b^4_1\in\widehat{T}_{a^1_1\otimes a^2_1\otimes a^3_1\otimes b^4_1} X_3$. If $b^3_3=(a^3_1+\mu a^3_2)\otimes b^4_3$, and if $\mu\neq 0$, let $\widetilde{a^3_2}=a^3_1+\mu a^3_2$, then $T=T'+(1/\mu a^1_1+a^1_3)\otimes a^2_1\otimes \widetilde{a^3_2}\otimes b^4_3$, where $T'=a^1_1\otimes a^2_1\otimes a^3_1\otimes [b^3_2(\alpha^3_1)-1/\mu (b^4_3+\lambda b^4_1)]+a^1_1\otimes a^2_1\otimes(\lambda/\mu \widetilde {a^3_2})\otimes b^4_1+a^1_1\otimes a^2_2\otimes a^3_1\otimes b^4_1+a^1_2\otimes a^2_1\otimes a^3_1\otimes b^4_1\in\widehat{T}_{a^1_1\otimes a^2_1\otimes a^3_1\otimes b^4_1} X_3$. If $\mu=0$, $T=T'+T''$, where $T'=a^1_1\otimes a^2_1\otimes a^3_1\otimes b^3_2(\alpha^3_1)+a^1_1\otimes a^2_1\otimes(\lambda a^3_2)\otimes b^4_1+a^1_1\otimes a^2_2\otimes a^3_1\otimes b^4_1+a^1_2\otimes a^2_1\otimes a^3_1\otimes b^4_1\in\widehat{T}_{a^1_1\otimes a^2_1\otimes a^3_1\otimes b^4_1} X_3$, and $T''=a^1_1\otimes a^2_1\otimes a^3_2\otimes b^4_3+a^1_3\otimes a^2_1\otimes a^3_1\otimes b^4_3\in\widehat{T}_{a^1_1\otimes a^2_1\otimes a^3_1\otimes b^4_3} X_3$. If $b^4_1$ and $b^4_3$ are linearly dependent, say $b^4_1=b^4_3$, then $T=T'+T''$, where $T'=a^1_1\otimes a^2_1\otimes a^3_1\otimes b^3_2(\alpha^3_1)+a^1_1\otimes a^2_2\otimes a^3_1\otimes b^4_1+(a^1_2+xa^1_3)\otimes a^2_1\otimes a^3_1\otimes b^4_1\in\widehat{T}_{a^1_1\otimes a^2_1\otimes a^3_1\otimes b^4_1} X_3$, and $T''=a^1_1\otimes a^2_1\otimes a^3_2\otimes b^3_2(\alpha^3_2)+(ya^1_3)\otimes a^2_1\otimes a^3_2\otimes b^4_1\in\widehat{T}_{a^1_1\otimes a^2_1\otimes a^3_2\otimes b^4_1} X_3$.

When $b^2_3=a^2_2+\lambda a^2_1$, $\dim T(A^*_2\otimes A^*_3)\leq 3$ implies three cases. Case 1: $b^3_1=a^3_1\otimes b^4_1$ and $b^3_2=a^3_1\otimes b^4_2$ for some $b^4_1, b^4_2\in A_4\otimes\cdots\otimes A_n$; Case 2: $b^3_1=a^3_1\otimes b^4_1$ and $b^3_3=a^3_1\otimes b^4_3$ for some $b^4_1, b^4_3\in A_4\otimes\cdots\otimes A_n$; Case 3: $b^3_1=a^3_1\otimes b^4_1$ and $b^3_3=a^3_2\otimes b^4_3$ for some $b^4_1, b^4_3\in A_4\otimes\cdots\otimes A_n$.

For case 1, if $b^3_3=u^3_3\otimes u^4_3$ for some $u^3_3\in A_3$ and $u^4_3\in A_4\otimes\cdots\otimes A_n$, then $T=T'+a^1_3\otimes (a^2_2+\lambda a^2_1)\otimes u^3_3\otimes u^4_3$, where $T'=a^1_1\otimes a^2_1\otimes a^3_1\otimes b^4_2+a^1_1\otimes a^2_2\otimes a^3_1\otimes b^4_1+a^1_2\otimes a^2_1\otimes a^3_1\otimes b^4_1\in\widehat{T}_{a^1_1\otimes a^2_1\otimes a^3_1\otimes b^4_1} X_3$. If $b^3_3: A^*_3\rightarrow A_4\otimes\cdots\otimes A_n$ has rank $2$, $\dim T(A^*_1\otimes A^*_3)\leq 3$ requires $b^4_1=b^4_2$ up to a scalar, and $b^4_1$ is a linear combination of $b^3_3(\alpha^3_1)$ and $b^3_3(\alpha^3_2)$, say $b^3_3(\alpha^3_1)=xb^3_3(\alpha^3_2)+yb^4_1$ or $b^4_1=b^3_3(\alpha^3_2)$ up to a scalar, then $T=(a^1_1+a^1_2+y\lambda a^1_3)\otimes a^2_1\otimes a^2_1\otimes a^3_1\otimes b^4_1+(a^1_1+ya^1_3)\otimes a^2_2\otimes a^3_1\otimes b^4_1+a^1_3\otimes (a^2_2+\lambda a^2_1)\otimes (xa^3_1+a^3_2)\otimes b^3_3(\alpha^3_2)$, or $T=T'+T''$, where $T'=a^1_1\otimes a^2_2\otimes a^3_1\otimes b^4_1+a^1_3\otimes a^2_2\otimes a^3_1\otimes b^3_3(\alpha^3_1)+a^1_3\otimes a^2_2\otimes a^3_2\otimes b^4_1\in\widehat{T}_{a^1_3\otimes a^2_2\otimes a^3_1\otimes b^4_1} X_3$, and $T''=(a^1_1+a^1_2)\otimes a^2_1\otimes a^3_1\otimes b^4_1+a^1_3\otimes a^2_1\otimes a^3_1\otimes \lambda b^3_3(\alpha^3_1)+a^1_3\otimes a^2_1\otimes (\lambda a^3_2)\otimes b^4_1\in\widehat{T}_{a^1_3\otimes a^2_1\otimes a^3_1\otimes b^4_1} X_3$.

For case 2, if $b^4_3=b^4_1$ up to a scalar, then $b^3_1=b^3_3$ up to a scalar, and $T=a^1_1\otimes a^2_1\otimes b^3_2+(a^1_1+a^1_3)\otimes a^2_2\otimes b^3_1+(a^1_2+\lambda a^1_3)\otimes a^2_1\otimes b^3_1$, which is discussed in \textbf{Case 2 Type 1}. Hence we assume $b^4_1$ and $b^4_3$ are linearly independent. $\dim T(A^*_1\otimes A^*_3)\leq 3$ implies $b^3_2(\alpha^3_2)=b^4_1$ up to a scalar, then $T=T'+a^1_3\otimes (a^2_2+\lambda a^2_1)\otimes a^3_1\otimes b^4_3$, where $T'=a^1_1\otimes a^2_1\otimes a^3_1\otimes b^3_2(\alpha^3_1)+a^1_1\otimes a^2_1\otimes a^3_2\otimes b^4_1+a^1_1\otimes a^2_2\otimes a^3_1\otimes b^4_1+a^1_2\otimes a^2_1\otimes a^3_1\otimes b^4_1\in\widehat{T}_{a^1_1\otimes a^2_1\otimes a^3_1\otimes b^4_1} X_3$.

For case 3, $\dim T(A^*_2\otimes A^*_3)\leq 3$ requires $b^3_2(\alpha^3_2)=b^4_1$ up to a scalar. Then $T=T'+a^1_3\otimes (a^2_2+\lambda a^2_1)\otimes a^3_2\otimes b^4_3$, where $T'=a^1_1\otimes a^2_1\otimes a^3_1\otimes b^3_2(\alpha^3_1)+a^1_1\otimes a^2_1\otimes a^3_2\otimes b^4_1+a^1_1\otimes a^2_2\otimes a^3_1\otimes b^4_1+a^1_2\otimes a^2_1\otimes a^3_1\otimes b^4_1\in\widehat{T}_{a^1_1\otimes a^2_1\otimes a^3_1\otimes b^4_1} X_3$.


Now we assume $T=\displaystyle\sum_{i=1}^k b^1_1\otimes\cdots\otimes b^{i-1}_1\otimes b^i_2\otimes b^{i+1}_1\otimes \cdots\otimes b^k_1+b^1_3\otimes\cdots\otimes b^k_3$. The induction argument is similar to the case $k=3$.


\textbf{Type 3}: $T=a^1_1\otimes b^2_2\otimes b^3_2+a^1_2\otimes b^2_1\otimes b^3_2+a^1_2\otimes b^2_2\otimes b^3_1+a^1_1\otimes b^2_1\otimes b^3_3+a^1_1\otimes b^2_3\otimes b^3_1+a^1_3\otimes b^2_1\otimes b^3_1$. Without loss of generality, we can assume $b^2_1=a^2_1$, $b^2_2=a^2_2$, and $b^2_3=xa^2_1+ya^2_2$. $\dim T(A^*_2\otimes A^*_3)\leq 3$ implies two cases. Case 1: $b^3_1=a^3_1\otimes b^4_1$ for some $b^4_1\in A_4\otimes\cdots\otimes A_n$, $b^3_2(\alpha^3_2)=b^4_1$ up to a scalar, and $b^3_2(\alpha^3_1)=b^3_3(\alpha^3_2)+\lambda b^4_1$ for some $\lambda\in\mathbb{C}$; Case 2: $b^3_1=a^3_1\otimes b^4_1$, and $b^3_2=a^3_1\otimes b^4_2$ for some $b^4_1, b^4_2\in A_4\otimes\cdots\otimes A_n$.

For case 1, $T=a^1_1\otimes a^2_2\otimes a^3_1\otimes b^3_3(\alpha^3_2)+a^1_1\otimes a^2_2\otimes a^3_2\otimes b^4_1+a^1_2\otimes a^2_1\otimes a^3_1\otimes b^3_3(\alpha^3_2)+a^1_2\otimes a^2_1\otimes a^3_2\otimes b^4_1+a^1_2\otimes a^2_2\otimes a^3_1\otimes b^4_1+a^1_1\otimes a^2_1\otimes a^3_2\otimes b^3_3(\alpha^3_2)+a^1_1\otimes a^2_1\otimes a^3_1\otimes b^3_3(\alpha^3_1)+a^1_1\otimes (y+\lambda)a^2_2\otimes a^3_1\otimes b^4_1+(xa^1_1+\lambda a^1_2+a^1_3)\otimes a^2_1\otimes a^3_1\otimes b^4_1$. Let $S(t)=[a^1_1+ta^1_2+t^2(x a^1_1+\lambda a^1_2+a^1_3)]\otimes[a^2_1+ta^2_2+t^2(y+\lambda) a^2_2]\otimes(a^3_1+ta^3_2)\otimes[b^4_1+tb^3_3(\alpha^3_2)+t^2b^3_3(\alpha^3_1)]$, then $T=S''(0)$.

For case 2, if $b^4_2=\lambda b^4_1$ for some $\lambda\in\mathbb{C}$, then $b^3_2=\lambda b^3_1$, $T=[(y+\lambda)a^1_1+a^1_2]\otimes a^2_2\otimes b^3_1+(x a^1_1+\lambda a^1_2+a^1_3)\otimes a^2_1\otimes b^3_1+a^1_1\otimes a^2_1\otimes b^3_3$, which is discussed in \textbf{Case 2 Type 1}. Thus we assume $b^4_1$ and $b^4_2$ are independent. $\dim T(A^*_1\otimes A^*_3)\leq 3$ implies $b^3_3(\alpha^3_2)=b^4_1$ up to a scalar, so $T=a^1_1\otimes a^2_2\otimes a^3_1\otimes b^4_2+a^1_2\otimes a^2_1\otimes a^3_1\otimes b^4_2+a^1_2\otimes a^2_2\otimes a^3_1\otimes b^4_1+a^1_1\otimes a^2_1\otimes a^3_1\otimes b^3_3(\alpha^3_1)+a^1_1\otimes a^2_1\otimes a^3_2\otimes b^4_1+a^1_1\otimes (xa^2_1+ya^2_2)\otimes a^3_1\otimes b^4_1+a^1_3\otimes a^2_1\otimes a^3_1\otimes b^4_1$. Let $S(t)=[a^1_1+ta^1_2+t^2a^1_3]\otimes[a^2_1+ta^2_2+t^2(xa^2_1+ya^2_2)]\otimes(a^3_1+t^2a^3_2)\otimes[b^4_1+tb^4_2+t^2b^3_3(\alpha^3_1)]$, then $T=S''(0)$.


Now we assume $T=\displaystyle\sum_{i<j} b^1_1\otimes\cdots\otimes b^{i-1}_1\otimes b^i_2\otimes b^{i+1}_1\otimes \cdots\otimes b^{j-1}_1\otimes b^j_2\otimes b^{j+1}_1\otimes\cdots\otimes b^k_1+\displaystyle\sum_{i=1}^k b^1_1\otimes\cdots\otimes b^{i-1}_1\otimes b^i_3\otimes b^{i+1}_1\otimes \cdots\otimes b^k_1$, and use induction to show $T \in \sigma_3(X)$. The induction argument is similar to the case $k=3$.


\textbf{Type 4}: $T=a^1_2\otimes b^2_1\otimes b^3_2+a^1_2\otimes b^2_2\otimes b^3_1+a^1_1\otimes b^2_1\otimes b^3_3+a^1_1\otimes b^2_3\otimes b^3_1+a^1_3\otimes b^2_1\otimes b^3_1$. If $b^2_2=b^2_1$, $T=a^1_2\otimes b^2_1\otimes b^3_2+a^1_1\otimes b^2_1\otimes b^3_3+a^1_1\otimes b^2_3\otimes b^3_1+(a^1_2+a^1_3)\otimes b^2_1\otimes b^3_1$, which is discussed in \textbf{Case 2 Type 2}. Hence we can assume $b^2_i=a^2_i$ for $1\leq i\leq 2$. Assume $b^2_3=xa^2_1+ya^2_2$, then $T=(ya^1_1+a^1_2)\otimes a^2_1\otimes b^3_2+(ya^1_1+a^1_2)\otimes a^2_2\otimes b^3_1+a^1_1\otimes a^2_1\otimes (b^3_3-yb^3_2)+(xa^1_1+a^1_3)\otimes a^2_1\otimes b^3_1$. Therefore after a change of basis, we only need to consider the case $T=a^1_2\otimes a^2_1\otimes b^3_2+a^1_2\otimes a^2_2\otimes b^3_1+a^1_1\otimes a^2_1\otimes b^3_3+a^1_3\otimes a^2_1\otimes b^3_1$. $\dim T(A^*_2\otimes A^*_3)\leq 3$ implies $b^3_1: A^*_3\rightarrow A_4\otimes\cdots\otimes A_n$ has rank $1$, say $b^3_1=a^3_1\otimes b^4_1$ for some $b^4_1\in A_4\otimes\cdots\otimes A_n$.

If $b^3_3(\alpha^3_1)$ and $b^3_3(\alpha^3_2)$ are linearly independent, $\dim T(A^*_1\otimes A^*_3)\leq 3$ implies $b^4_1, b^3_2(\alpha^3_2)$ are in $V_4$, where $V_4$ is spanned by $b^3_3(\alpha^3_1)$ and $b^3_3(\alpha^3_2)$. For the subspace $A_3\otimes V_4$, after a change of basis, we can assume $a^3_1$ and $a^3_1\otimes b^3_3(\alpha^3_1)+a^3_2\otimes b^3_3(\alpha^3_2)$ are preserved, and $b^4_1 = b^3_3(\alpha^3_1)$, or $b^4_1 = b^3_3(\alpha^3_2)$. So we have two cases:

Case 1: If $b^4_1=b^3_3(\alpha^3_1)$, assume $b^3_2(\alpha^3_2)=xb^3_3(\alpha^3_1)+yb^3_3(\alpha^3_2)$, then $T=T'+(ya^1_2+a^1_1)\otimes a^2_1\otimes a^3_2\otimes b^3_3(\alpha^3_2)$, where $T'=a^1_2\otimes a^2_1\otimes a^3_1\otimes b^3_2(\alpha^3_1)+a^1_2\otimes a^2_1\otimes (xa^3_2)\otimes b^3_3(\alpha^3_1)+a^1_2\otimes a^2_2\otimes a^3_1\otimes b^3_3(\alpha^3_1)+(a^1_1+a^1_3)\otimes a^2_1\otimes a^3_1\otimes b^3_3(\alpha^3_1)\in\widehat{T}_{a^1_2\otimes a^2_1\otimes a^3_1\otimes b^3_3(\alpha^3_1)} X_3$.

Case 2: If $b^4_1=b^3_3(\alpha^3_2)$, we can assume $b^3_2(\alpha^3_2)=b^3_3(\alpha^3_1)+\lambda b^3_3(\alpha^3_2)$ for some $\lambda\in\mathbb{C}$, or $b^3_2(\alpha^3_2)=\lambda b^3_3(\alpha^3_2)$. If $b^3_2(\alpha^3_2)=b^3_3(\alpha^3_1)+\lambda b^3_3(\alpha^3_2)$, $T=a^1_2\otimes a^2_1\otimes a^3_1\otimes b^3_2(\alpha^3_1)+a^1_2\otimes a^2_1\otimes (\lambda a^3_2)\otimes b^3_3(\alpha^3_2)+a^1_2\otimes a^2_2\otimes a^3_1\otimes b^3_3(\alpha^3_2)+a^1_3\otimes a^2_1\otimes a^3_1\otimes b^3_3(\alpha^3_2)+a^1_2\otimes a^2_1\otimes a^3_2\otimes b^3_3(\alpha^3_1)+a^1_1\otimes a^2_1\otimes a^3_1\otimes b^3_3(\alpha^3_1)+a^1_1\otimes a^2_1\otimes a^3_2\otimes b^3_3(\alpha^3_2)$. Let $S(t)=(a^1_2+ta^1_1+t^2a^1_3)\otimes(a^2_1+t^2a^2_2)\otimes(a^3_1+ta^3_2+t^2\lambda a^3_2)\otimes(b^3_3(\alpha^3_2)+tb^3_3(\alpha^3_1)+t^2b^3_2(\alpha^3_1))$, then $T=S''(0)$. If $b^3_2(\alpha^3_2)=\lambda b^3_3(\alpha^3_2)$, $T=T'+T''$, where $T'=a^1_2\otimes a^2_1\otimes a^3_1\otimes b^3_2(\alpha^3_1)+a^1_2\otimes a^2_1\otimes \lambda a^3_2\otimes b^3_3(\alpha^3_2)+a^1_2\otimes a^2_2\otimes a^3_1\otimes b^3_3(\alpha^3_2)+a^1_3\otimes a^2_1\otimes a^3_1\otimes b^3_3(\alpha^3_2)\in\widehat{T}_{a^1_2\otimes a^2_1\otimes a^3_1\otimes b^3_3(\alpha^3_2)} X_3$, and $T''=a^1_1\otimes a^2_1\otimes a^3_1\otimes b^3_3(\alpha^3_1)+a^1_1\otimes a^2_1\otimes a^3_2\otimes b^3_3(\alpha^3_2)\in\widehat{T}_{a^1_1\otimes a^2_1\otimes a^3_1\otimes b^3_3(\alpha^3_2)} X_3$.

If $b^3_3(\alpha^3_2)=\lambda b^3_3(\alpha^3_1)$ for some $\lambda\in\mathbb{C}$, then we can assume $b^3_3=a^3_1\otimes b^3_3(\alpha^3_1)$ or $b^3_3=a^3_2\otimes b^3_3(\alpha^3_1)$. Thus we have four cases:

Case 1: If $b^3_3=a^3_1\otimes b^3_3(\alpha^3_1)$, $b^3_3(\alpha^3_1)$ and $b^4_1$ are linearly independent, we can assume $b^3_2(\alpha^3_2)=xb^4_1+yb^3_3(\alpha^3_1)$ for some $x, y\in\mathbb{C}$ due to $\dim T(A^*_1\otimes A^*_3)$, then $T=T'+T''$, where $T'=a^1_2\otimes a^2_1\otimes a^3_1\otimes b^3_2(\alpha^3_1)+a^1_2\otimes a^2_1\otimes xa^3_2\otimes b^4_1+a^1_2\otimes a^2_2\otimes a^3_1\otimes b^4_1+a^1_3\otimes a^2_1\otimes a^3_1\otimes b^4_1\in\widehat{T}_{a^1_2\otimes a^2_1\otimes a^3_1\otimes b^4_1} X_3$, and $T''=a^1_2\otimes a^2_1\otimes ya^3_2\otimes b^3_3(\alpha^3_1)+a^1_1\otimes a^2_1\otimes a^3_1\otimes b^3_3(\alpha^3_1)\in\widehat{T}_{a^1_2\otimes a^2_1\otimes a^3_1\otimes b^3_3(\alpha^3_1)} X_3$. 

Case 2: If $b^3_3=a^3_1\otimes b^3_3(\alpha^3_1)$ and $b^3_3(\alpha^3_1)=\mu b^4_1$ for some $\mu\in\mathbb{C}$, $T=T'+a^1_2\otimes a^2_1\otimes a^3_2\otimes b^3_2(\alpha^3_2)$, where $T'=a^1_2\otimes a^2_1\otimes a^3_1\otimes b^3_2(\alpha^3_1)+a^1_2\otimes a^2_2\otimes a^3_1\otimes b^4_1+(\mu a^1_1+a^1_3)\otimes a^2_1\otimes a^3_1\otimes b^4_1\in\widehat{T}_{a^1_2\otimes a^2_1\otimes a^3_1\otimes b^4_1} X_3$.

Case 3: If $b^3_3=a^3_2\otimes b^3_3(\alpha^3_1)$, $b^3_3(\alpha^3_1)$ and $b^4_1$ are linearly independent, we can assume $b^3_2(\alpha^3_2)=xb^4_1+yb^3_3(\alpha^3_1)$ due to $\dim T(A^*_1\otimes A^*_3)$, then $T=T'+(ya^1_2+a^1_1)\otimes a^2_1\otimes a^3_2\otimes b^3_3(\alpha^3_1)$, where $T'=a^1_2\otimes a^2_1\otimes a^3_1\otimes b^3_2(\alpha^3_1)+a^1_2\otimes a^2_1\otimes xa^3_2\otimes b^4_1+a^1_2\otimes a^2_2\otimes a^3_1\otimes b^4_1+a^1_3\otimes a^2_1\otimes a^3_1\otimes b^4_1\in\widehat{T}_{a^1_2\otimes a^2_1\otimes a^3_1\otimes b^4_1} X_3$.

Case 4: If $b^3_3=a^3_2\otimes b^3_3(\alpha^3_1)$ and $b^3_3(\alpha^3_1)=\mu b^4_1$ for some $\mu\in\mathbb{C}$, $T=T'+T''$, where $T'=a^1_2\otimes a^2_1\otimes a^3_1\otimes b^3_2(\alpha^3_1)+a^1_2\otimes a^2_2\otimes a^3_1\otimes b^4_1+a^1_3\otimes a^2_1\otimes a^3_1\otimes b^4_1\in\widehat{T}_{a^1_2\otimes a^2_1\otimes a^3_1\otimes b^4_1} X_3$, and $T''=a^1_1\otimes a^2_1\otimes a^3_2\otimes \mu b^4_1+a^1_2\otimes a^2_1\otimes a^3_2\otimes b^3_2(\alpha^3_2)\in\widehat{T}_{a^1_2\otimes a^2_1\otimes a^3_2\otimes b^4_1} X_3$.


Now assume $T=\displaystyle\sum_{i=2}^k b^1_2\otimes b^2_1\otimes\cdots\otimes b^{i-1}_1\otimes b^i_2\otimes b^{i+1}_1\otimes\cdots\otimes b^k_1+\sum_{i=1}^k b^1_1\otimes\cdots\otimes b^{i-1}_1\otimes b^i_3\otimes b^{i+1}_1\otimes\cdots\otimes b^k_1$, and use induction to show $T \in \sigma_3(X)$. The induction argument is similar to the case $k=3$.

\subsection{Case 3: $T\in\sigma_3(X_2)\setminus\sigma_2(X_2)$, $T\in Sub_{2,2,\dots,2}(A_1\otimes \cdots\otimes A_n)$}

Since $\dim T(A^*_3\otimes\cdots\otimes A^*_n)\leq 3$, then after a change of basis we can assume $T(A^*_3\otimes\cdots\otimes A^*_n)\subset V$, where $V$ is spanned by $\{a^1_1\otimes a^2_1+a^1_2\otimes a^2_2, a^1_1\otimes a^2_2, a^1_2\otimes a^2_1\}$ or $\{a^1_1\otimes a^2_1, a^1_1\otimes a^2_2, a^1_2\otimes a^2_1\}$. So $T$ has $2$ types of normal forms.

\textbf{Type 1}: $T=(a^1_1\otimes a^2_1+a^1_2\otimes a^2_2)\otimes b^3_1+a^1_1\otimes a^2_2\otimes b^3_2+a^1_2\otimes a^2_1\otimes b^3_3$, we reduce the problem to finding equations for $\sigma_3(\nu_n(\mathbb{P}^1))$, which has been settled.

\begin{lemma}
Let $T\in A\otimes B\otimes C$, where $\dim A=\dim B$. If there is an element $\phi\in Ker(T^{\wedge}_{BA})$ with full rank, then $\phi(T)\in S^2A\otimes C$.
\end{lemma}

\begin{proof}[Proof of the lemma]

Let $\{a_i\}$, $\{b_j\}$, $\{c_k\}$ be bases for $A$, $B$, $C$ respectively, and $\{a^i\}$, $\{b^j\}$, $\{c^k\}$ their dual bases. Let $T=\sum \alpha^{ijk} a_i\otimes b_j\otimes c_k$, then $T^{\wedge}_{BA}: a_l\otimes b^j\mapsto\sum_{i, k}\alpha^{ijk} (a_l\wedge a_i)\otimes c_k$. Let $\phi=\sum \beta^l_j a_l\otimes b^j\in Ker(T^{\wedge}_{BA})$, then $\sum \beta^l_j\alpha^{ijk}(a_l\wedge a_i)\otimes c_k=0$, which means $\sum_j \beta^l_j\alpha^{ijk}=\sum_j \beta^i_j\alpha^{ljk}$. Since $\phi(T)=\sum \beta^l_j\alpha^{ijk}a_i\otimes a_l\otimes c_k$, then $\phi(T)\in S^2A\otimes C$.

\end{proof}

Let $V$ be a complex vector space. Given $\phi\in S^dV$, let $\phi_{a,d-a}\in S^aV\otimes S^{d-a}V$ denote the $(a,d-a)$-polarization of $\phi$. As a linear map $S^aV^*\rightarrow S^{d-a}V$, $\operatorname{rank}(\phi_{a,d-a})\leq r$ if $[\phi]\in\sigma_r(\nu_d(\mathbb{P}V))$ \cite{LO}.

\begin{theorem}[\cite{LO}]
\label{thm:veronese}
$\sigma_3(\nu_3(\mathbb{P}^n))$ is ideal theoretically defined by Aronhold invariant and size $4$ minors of $\phi_{1,2}$. $\sigma_3(\nu_d(\mathbb{P}^n))$ is scheme theoretically defined by size $4$ minors of $\phi_{2,2}$ and $\phi_{1,3}$ when $d\geq 4$.
\end{theorem}

Now given any $T\in A_1\otimes \cdots\otimes A_n$, if there is some $1\leq i\leq n$, and for any $j\neq i$, there is a $\phi_{ji}\in Ker(T^{\wedge}_{A_jA_i})$ with full rank, then $\widetilde{T}=\phi_{ni}\circ\cdots\circ\phi_{1i}(T)\in S^nA_i$ has the same rank with $T$. If $T$ satisfies $4\times 4$ minors of flattenings, $\widetilde{T}$ satisfies size $4$ minors of symmetric flattenings, by Theorem~\ref{thm:veronese} $\widetilde{T}\in\sigma_3(\nu_n(\mathbb{P}^1))$, then $T\in\sigma_3 (X)$. If $T$ is of Type 1, we always have $a^1_1\otimes a^2_2+a^1_2\otimes a^2_1\in Ker(T^{\wedge}_{A_2A_1})$ with full rank, hence if for any $2\leq i\leq n$, $T$ is of Type 1 when viewed as a tensor in $A_1\otimes A_i\otimes (A_2\otimes\cdots\otimes A_{i-1}\otimes\widehat{A_i}\otimes A_{i+1}\otimes\cdots\otimes A_n)$, then $T\in\sigma_3 (X)$. If $T\in A_1\otimes A_2\otimes (A_3\otimes\cdots\otimes A_n)$ is not of Type 1, then it must be of Type 2, and we will use induction to show that $T\in\sigma_3 (X)$ in this situation.


\textbf{Type 2}: $T=a^1_1\otimes a^2_1\otimes b^3_1+a^1_1\otimes a^2_2\otimes b^3_2+a^1_2\otimes a^2_1\otimes b^3_3$, the dimension of $T(A^*_2\otimes A^*_3)$ implies $b^3_3: A^*_3\rightarrow A_4\otimes \cdots\otimes A_n$ has rank $1$, or $b^3_2: A^*_3\rightarrow A_4\otimes\cdots\otimes A_n$ has rank $1$.

If $b^3_3: A^*_3\rightarrow A_4\otimes\cdots\otimes A_n$ has rank $1$, say $b^3_3=a^3_1\otimes b^4_3$, and $b^3_2: A^*_3\rightarrow A_4\otimes \cdots\otimes A_n$ has rank $2$, say $b^3_2=a^3_1\otimes b^4_1+a^3_2\otimes b^4_2$, then $\dim T(A^*_2\otimes A^*_3)\leq 3$ implies $b^3_1(\alpha^3_2)=\lambda b^4_1+\mu b^4_2$ for some $\lambda, \mu\in\mathbb{C}$. If $b^4_3$, $b^4_1$ and $b^4_2$ are linearly independent, then $\dim T(A^*_1\otimes A^*_2\otimes A^*_3)\leq 3$ forces $b^3_1(\alpha^3_1)=xb^4_3+yb^4_1+zb^4_2$ for some $x, y, z\in\mathbb{C}$, thus $T=a^1_1\otimes a^2_1\otimes (ya^3_1\otimes b^4_1+za^3_1\otimes b^4_2+\lambda a^3_2\otimes b^4_1+\mu a^3_2\otimes b^4_2)+a^1_1\otimes a^2_2\otimes (a^3_1\otimes b^4_1+a^3_2\otimes b^4_2)+(xa^1_1+a^1_2)\otimes a^2_1\otimes a^3_1\otimes b^4_3$. For the subspace $A_3\otimes V_4$, where $V_4\subset A_4\otimes\cdots\otimes A_n$ is spanned by $b^4_1$ and $b^4_2$, after a change of basis we can assume $a^3_1\otimes b^4_1+a^3_2\otimes b^4_2$ is preserved, $a^3_1$ is mapped to $ua^3_1+va^3_2$ for some $u, v\in\mathbb{C}$, and $ya^3_1\otimes b^4_1+za^3_1\otimes b^4_2+\lambda a^3_2\otimes b^4_1+\mu a^3_2\otimes b^4_2$ is of the Jordan canonical form, i.e. $a^3_1\otimes b^4_1+a^3_2\otimes b^4_2$, or $a^3_1\otimes b^4_1$, or $a^3_1\otimes b^4_2$, or $\beta a^3_1\otimes b^4_1+a^3_1\otimes b^4_2+\beta a^3_2\otimes b^4_2$ for some $0\neq\beta\in\mathbb{C}$. Hence we have:

Subcase 1: $T=a^1_1\otimes (a^2_1+a^2_2)\otimes a^3_1\otimes b^4_1+a^1_1\otimes (a^2_1+a^2_2)\otimes a^3_2\otimes b^4_2+(xa^1_1+a^1_2)\otimes a^2_1\otimes (ua^3_1+va^3_2)\otimes b^4_3$.

Subcase 2: $T=a^1_1\otimes (a^2_1+a^2_2)\otimes a^3_1\otimes b^4_1+a^1_1\otimes a^2_2\otimes a^3_2\otimes b^4_2+(xa^1_1+a^1_2)\otimes a^2_1\otimes (ua^3_1+va^3_2)\otimes b^4_3$.

Subcase 3: $T=T'+(xa^1_1+a^1_2)\otimes a^2_1\otimes (ua^3_1+va^3_2)\otimes b^4_3$, where $T'=a^1_1\otimes a^2_1\otimes a^3_1\otimes b^4_2+a^1_1\otimes a^2_2\otimes a^3_1\otimes b^4_1+a^1_1\otimes a^2_2\otimes a^3_2\otimes b^4_2\in\widehat{T}_{a^1_1\otimes a^2_2\otimes a^3_1\otimes b^4_2} X_3$.

Subcase 4: $T=T'+(xa^1_1+a^1_2)\otimes a^2_1\otimes (ua^3_1+va^3_2)\otimes b^4_3$, where $T'=a^1_1\otimes (\beta a^2_1+a^2_2)\otimes a^3_1\otimes b^4_1+a^1_1\otimes (\beta a^2_1+a^2_2)\otimes a^3_2\otimes b^4_2+a^1_1\otimes a^2_1\otimes a^3_1\otimes b^4_2\in\widehat{T}_{a^1_1\otimes (\beta a^2_1+a^2_2)\otimes a^3_1\otimes b^4_2} X_3$.

If $b^4_3=pb^4_1+qb^4_2$ for some $p, q\in\mathbb{C}$, for $A_3\otimes V_4$, after a change of basis we can assume $a^3_1$ and $a^3_1\otimes b^4_1+a^3_2\otimes b^4_2$ are preserved, $b^4_3=b^4_1$ or $b^4_2$, and $a^3_2\otimes b^3_1(\alpha^3_2)$ is of the form $x^1_1a^3_1\otimes b^4_1+x^1_2a^3_1\otimes b^4_2+x^2_1a^3_2\otimes b^4_1+x^2_2a^3_2\otimes b^4_2$. If $b^4_3=b^4_1$ we have:

Subcase 5: $T=T'+a^1_1\otimes (x^2_2a^2_1+a^2_2)\otimes a^3_2\otimes b^4_2$, where $T'=a^1_1\otimes a^2_1\otimes a^3_1\otimes[b^3_1(\alpha^3_1)+x^1_1b^4_1+x^1_2b^4_2]+a^1_1\otimes a^2_1\otimes (x^2_1a^3_2)\otimes b^4_1+a^1_1\otimes a^2_2\otimes a^3_1\otimes b^4_1+a^1_2\otimes a^2_1\otimes a^3_1\otimes b^4_1\in\widehat{T}_{a^1_1\otimes a^2_1\otimes a^3_1\otimes b^4_1} X_3$.

If $b^4_3=b^4_2$, by changing $a^3_2$, $b^4_2$ and $a^2_1$, we can assume $x^2_1=1$ or $0$. So we have:

Subcase 6: $T=a^1_1\otimes a^2_1\otimes a^3_1\otimes [b^3_1(\alpha^3_1)+x^1_1b^4_1+x^1_2b^4_2]+a^1_1\otimes a^2_1\otimes (x^2_2a^3_2)\otimes b^4_2+a^1_2\otimes a^2_1\otimes a^3_1\otimes b^4_2+a^1_1\otimes a^2_1\otimes a^3_2\otimes b^4_1+a^1_1\otimes a^2_2\otimes a^3_1\otimes b^4_1+a^1_1\otimes a^2_2\otimes a^3_2\otimes b^4_2$. Let $S(t)=(a^1_1+t^2a^1_2)\otimes(a^2_1+ta^2_2)\otimes(a^3_1+ta^3_2+t^2x^2_2a^3_2)\otimes[b^4_2+tb^4_1+t^2(b^3_1(\alpha^3_1)+x^1_1b^4_1+x^1_2b^4_2)]$, so $T=S''(0)$.

Subcase 7: $T=T'+T''$, where $T'=a^1_1\otimes a^2_1\otimes a^3_1\otimes [b^3_1(\alpha^3_1)+x^1_1b^4_1+x^1_2b^4_2]+a^1_1\otimes a^2_1\otimes (x^2_2a^3_2)\otimes b^4_2+a^1_2\otimes a^2_1\otimes a^3_1\otimes b^4_2\in\widehat{T}_{a^1_1\otimes a^2_1\otimes a^3_1\otimes b^4_2} X_3$, and $T''=a^1_1\otimes a^2_2\otimes a^3_1\otimes b^4_1+a^1_1\otimes a^2_2\otimes a^3_2\otimes b^4_2\in\widehat{T}_{a^1_1\otimes a^2_2\otimes a^3_1\otimes b^4_2} X_3$.

If $b^3_2: A^*_3\rightarrow A_4\otimes\cdots\otimes A_n$ has rank $1$, say $b^3_2=a^3_1\otimes b^4_2$ for some $b^4_2\in A_4\otimes\cdots\otimes A_n$, and $b^3_3: A^*_3\rightarrow A_4\otimes\cdots\otimes A_n$ has rank $2$, say $b^3_3=a^3_1\otimes b^4_1+a^3_2\otimes b^4_3$ for some $b^4_1, b^4_3\in A_4\otimes\cdots\otimes A_n$, then $\dim T(A^*_1\otimes A^*_3)\leq 3$ implies $b^3_1(\alpha^3_2)=\lambda b^4_1+\mu b^4_3$ for some $\lambda, \mu\in\mathbb{C}$. If $b^4_3$, $b^4_1$ and $b^4_2$ are linearly independent, then $\dim T(A^*_1\otimes A^*_2\otimes A^*_3)\leq 3$ forces $b^3_1(\alpha^3_1)=xb^4_1+yb^4_2+zb^4_3$ for some $x, y, z\in\mathbb{C}$. For the subspace $A_3\otimes V_4$, where $V_4\subset A_4\otimes\cdots\otimes A_n$ is spanned by $b^4_1$ and $b^4_3$, after a change of basis we can assume $a^3_1\otimes b^4_1+a^3_2\otimes b^4_3$ is preserved, $a^3_1$ is mapped to $ua^3_1+va^3_2$ for some $u, v\in\mathbb{C}$ under the new basis, and $xa^3_1\otimes b^4_1+za^3_1\otimes b^4_3+\lambda a^3_2\otimes b^4_1+\mu a^3_2\otimes b^4_3$ is of the Jordan canonical form, i.e. $a^3_1\otimes b^4_1+a^3_2\otimes b^4_3$, or $a^3_1\otimes b^4_1$, or $a^3_1\otimes b^4_3$, or $\beta a^3_1\otimes b^4_1+a^3_1\otimes b^4_3+\beta a^3_2\otimes b^4_3$ for some $0\neq\beta\in\mathbb{C}$. Hence we have:

Subcase 8: $T=(a^1_1+a^1_2)\otimes a^2_1\otimes a^3_1\otimes b^4_1+(a^1_1+a^1_2)\otimes a^2_1\otimes a^3_2\otimes b^4_3+a^1_1\otimes (ya^2_1+a^2_2)\otimes (ua^3_1+va^3_2)\otimes b^4_2$.

Subcase 9: $T=(a^1_1+a^1_2)\otimes a^2_1\otimes a^3_1\otimes b^4_1+a^1_2\otimes a^2_1\otimes a^3_2\otimes b^4_3+a^1_1\otimes (ya^2_1+a^2_2)\otimes (ua^3_1+va^3_2)\otimes b^4_2$.

Subcase 10: $T=T'+a^1_1\otimes (ya^2_1+a^2_2)\otimes (ua^3_1+va^3_2)\otimes b^4_2$, where $T'=a^1_1\otimes a^2_1\otimes a^3_1\otimes b^4_3+a^1_2\otimes a^2_1\otimes a^3_1\otimes b^4_1+a^1_2\otimes a^2_1\otimes a^3_2\otimes b^4_3\in\widehat{T}_{a^1_2\otimes a^2_1\otimes a^3_1\otimes b^4_3} X_3$.

Subcase 11: $T=T'+a^1_1\otimes (ya^2_1+a^2_2)\otimes (ua^3_1+va^3_2)\otimes b^4_2$, where $T'=(\beta a^1_1+a^1_2)\otimes a^2_1\otimes a^3_1\otimes b^4_1+(\beta a^1_1+a^1_2)\otimes a^2_1\otimes a^3_2\otimes b^4_3+a^1_1\otimes a^2_1\otimes a^3_1\otimes b^4_3\in\widehat{T}_{(\beta a^1_1+a^1_2)\otimes a^2_1\otimes a^3_1\otimes b^4_3} X_3$.

If $b^4_2=pb^4_1+qb^4_3$ for some $p, q\in\mathbb{C}$, for $A_3\otimes V_4$, after a change of basis we can assume $a^3_1$ and $a^3_1\otimes b^4_1+a^3_2\otimes b^4_3$ are preserved, $b^4_2=b^4_1$ or $b^4_3$, and $a^3_2\otimes b^3_1(\alpha^3_2)$ is of the form $x^1_1a^3_1\otimes b^4_1+x^1_2a^3_1\otimes b^4_3+x^2_1a^3_2\otimes b^4_1+x^2_2a^3_2\otimes b^4_3$. If $b^4_2=b^4_1$ we have:

Subcase 12: $T=T'+(x^2_2a^1_1+a^1_2)\otimes a^2_1\otimes a^3_2\otimes b^4_3$, where $T'=a^1_1\otimes a^2_1\otimes a^3_1\otimes [b^3_1(\alpha^3_1)+x^1_1b^4_1+x^1_2b^4_3]+a^1_1\otimes a^2_1\otimes x^2_1a^3_2\otimes b^4_1+a^1_1\otimes a^2_2\otimes a^3_1\otimes b^4_1+a^1_2\otimes a^2_1\otimes a^3_1\otimes b^4_1\in\widehat{T}_{a^1_1\otimes a^2_1\otimes a^3_1\otimes b^4_1} X_3$.

If $b^4_2=b^4_3$, by changing $a^3_2$, $b^4_3$ and $a^2_2$, we can assume $x^2_1=1$ or $0$. So we have:

Subcase 13: $T=a^1_1\otimes a^2_1\otimes a^3_1\otimes [b^3_1(\alpha^3_1)+x^1_1b^4_1+x^1_2b^4_3]+a^1_1\otimes a^2_1\otimes (x^2_2a^3_2)\otimes b^4_3+a^1_1\otimes a^2_2\otimes a^3_1\otimes b^4_3+a^1_2\otimes a^2_1\otimes a^3_1\otimes b^4_1+a^1_1\otimes a^2_1\otimes a^3_2\otimes b^4_1+a^1_2\otimes a^2_1\otimes a^3_2\otimes b^4_3$. Let $S(t)=(a^1_1+ta^1_2)\otimes(a^2_1+t^2a^2_2)\otimes(a^3_1+ta^3_2+t^2x^2_2a^3_2)\otimes[b^4_3+tb^4_1+t^2(b^3_1(\alpha^3_1)+x^1_1b^4_1+x^1_2b^4_3)]$, so $T=S''(0)$.

Subcase 14: $T=T'+T''$, where $T'=a^1_1\otimes a^2_1\otimes a^3_1\otimes [b^3_1(\alpha^3_1)+x^1_1b^4_1+x^1_2b^4_3]+a^1_1\otimes a^2_1\otimes (x^2_2a^3_2)\otimes b^4_3+a^1_1\otimes a^2_2\otimes a^3_1\otimes b^4_3\in\widehat{T}_{a^1_1\otimes a^2_1\otimes a^3_1\otimes b^4_3} X_3$, and $T''=a^1_2\otimes a^2_1\otimes a^3_1\otimes b^4_1+a^1_2\otimes a^2_1\otimes a^3_2\otimes b^4_3\in\widehat{T}_{a^1_2\otimes a^2_1\otimes a^3_1\otimes b^4_3} X_3$.

If both $b^3_2$ and $b^3_3: A^*_3\rightarrow A_4\otimes\cdots\otimes A_n$ have rank $1$, say $b^3_2=a^3_1\otimes b^4_2$ and $b^3_3=u^3_3\otimes b^4_3$ for some $u^3_3\in A_3$ and $b^4_2, b^4_3\in A_4\otimes\cdots\otimes A_n$, and $b^3_1: A^*_3\rightarrow A_4\otimes\cdots\otimes A_n$ has rank $2$, say $b^3_1=a^3_1\otimes u^4_1+a^3_2\otimes u^4_2$ for some $u^4_1, u^4_2\in A_4\otimes\cdots\otimes A_n$, $b^4_2$, $u^4_1$ and $u^4_2$ are linearly independent, then $b^4_3=xb^4_2+yu^4_1+zu^4_2$ for some $x, y, z\in\mathbb{C}$. After a change of basis, we can assume $x=0$ or $1$, $u^3_3=a^3_1$ or $a^3_2$. For the subspace $A_3\otimes V_4$, where $V_4$ is spanned by $u^4_1$ and $u^4_2$, after a change of basis we can assume $a^3_1\otimes u^4_1+a^3_2\otimes u^4_2$ and $a^3_1$ are preserved, and $yu^4_1+zu^4_2=u^4_1$ or $u^4_2$. Then we have:

Subcase 15: If $u^3_3=a^3_1$, $x=0$, $yu^4_1+zu^4_2=u^4_1$, then $T=(a^1_1+a^1_2)\otimes a^2_1\otimes a^3_1\otimes u^4_1+a^1_1\otimes a^2_1\otimes a^3_2\otimes u^4_2+a^1_1\otimes a^2_2\otimes a^3_1\otimes b^4_2$.

Subcase 16: If $u^3_3=a^3_1$, $x=0$, $yu^4_1+zu^4_2=u^4_2$, then $T=T'+a^1_1\otimes a^2_2\otimes a^3_1\otimes b^4_2$, where $T'=a^1_1\otimes a^2_1\otimes a^3_1\otimes u^4_1+a^1_1\otimes a^2_1\otimes a^3_2\otimes u^4_2+a^1_2\otimes a^2_1\otimes a^3_1\otimes u^4_2\in\widehat{T}_{a^1_1\otimes a^2_1\otimes a^3_1\otimes u^4_2} X_3$.

Subcase 17: If $u^3_3=a^3_1$, $x=1$, $yu^4_1+zu^4_2=u^4_1$, then $T=(a^1_1+a^1_2)\otimes a^2_1\otimes a^3_1\otimes (u^4_1+b^4_2)+a^1_1\otimes a^2_1\otimes a^3_2\otimes u^4_2+a^1_1\otimes (a^2_2-a^2_1)\otimes a^3_1\otimes b^4_2$.

Subcase 18: If $u^3_3=a^3_1$, $x=1$, $yu^4_1+zu^4_2=u^4_2$, then $T=T'+T''$, where $T'=a^1_1\otimes a^2_1\otimes a^3_1\otimes u^4_1+a^1_1\otimes a^2_2\otimes a^3_1\otimes b^4_2+a^1_2\otimes a^2_1\otimes a^3_1\otimes b^4_2\in\widehat{T}_{a^1_1\otimes a^2_1\otimes a^3_1\otimes b^4_2} X_3$, and $T''=a^1_1\otimes a^2_1\otimes a^3_2\otimes u^4_2+a^1_2\otimes a^2_1\otimes a^3_1\otimes u^4_2\in\widehat{T}_{a^1_1\otimes a^2_1\otimes a^3_1\otimes u^4_2} X_3$.

If $u^3_3=a^3_2$, for the subspace $A_3\otimes V_4$, after a change of basis we can assume $a^3_1\otimes u^4_1+a^3_2\otimes u^4_2$ and $a^3_2$ are preserved, $yu^4_1+zu^4_2=u^4_1$ or $u^4_2$, and $a^3_1$ is mapped to $\lambda a^3_1+\mu a^3_2$ for some $\lambda, \mu\in\mathbb{C}$ under the new basis. Then we have:

Subcase 19: If $x=0$, $yu^4_1+zu^4_2=u^4_1$, then $T=T'+a^1_1\otimes a^2_2\otimes(\lambda a^3_1+\mu a^3_2)\otimes b^4_2$, where $T'=a^1_1\otimes a^2_1\otimes a^3_1\otimes u^4_1+a^1_1\otimes a^2_1\otimes a^3_2\otimes u^4_2+a^1_2\otimes a^2_1\otimes a^3_2\otimes u^4_1\in\widehat{T}_{a^1_1\otimes a^2_1\otimes a^3_2\otimes u^4_1} X_3$.

Subcase 20: If $x=0$, $yu^4_1+zu^4_2=u^4_2$, then $T=a^1_1\otimes a^2_1\otimes a^3_1\otimes u^4_1+(a^1_1+a^1_2)\otimes a^2_1\otimes a^3_2\otimes u^4_2+a^1_1\otimes a^2_2\otimes (\lambda a^3_1+\mu a^3_2)\otimes b^4_2$.

By adjusting $a^2_2$, we can assume $\lambda a^3_1+\mu a^3_2=a^3_2$ or $a^3_1+\mu a^3_2$. So we have:

Subcase 21: If $\lambda a^3_1+\mu a^3_2=a^3_2$, $x=1$, $yu^4_1+zu^4_2=u^4_1$, $T=T'+T''$, where $T'=a^1_1\otimes a^2_1\otimes a^3_2\otimes u^4_2+a^1_1\otimes a^2_2\otimes a^3_2\otimes b^4_2+a^1_2\otimes a^2_1\otimes a^3_2\otimes b^4_2\in\widehat{T}_{a^1_1\otimes a^2_1\otimes a^3_2\otimes b^4_2} X_3$, and $T''=a^1_1\otimes a^2_1\otimes a^3_1\otimes u^4_1+a^1_2\otimes a^2_1\otimes a^3_2\otimes u^4_1\in\widehat{T}_{a^1_1\otimes a^2_1\otimes a^3_2\otimes u^4_1} X_3$.

Subcase 22: If $\lambda a^3_1+\mu a^3_2=a^3_2$, $x=1$, $yu^4_1+zu^4_2=u^4_2$, $T=(a^1_1+a^1_2)\otimes a^2_1\otimes a^3_2\otimes (u^4_2+b^4_2)+a^1_1\otimes (a^2_2-a^2_1)\otimes a^3_2\otimes b^4_2+a^1_1\otimes a^2_1\otimes a^3_1\otimes u^4_1$.

Subcase 23: If $\lambda a^3_1+\mu a^3_2=a^3_1+\mu a^3_2$, $x=1$, $yu^4_1+zu^4_2=u^4_1$, let $c^2_1=a^2_1$, $c^2_2=a^2_2-a^2_1$, $v^4_1=u^4_1+b^4_2$ and $v^4_2=b^4_2$, then $T=a^1_1\otimes a^3_1\otimes(c^2_1\otimes v^4_1+c^2_2\otimes v^4_2)+a^1_1\otimes a^3_2\otimes (c^2_1\otimes u^4_2+\mu c^2_1\otimes v^4_2+\mu c^2_2\otimes v^4_2)+a^1_2\otimes a^3_2\otimes c^2_1\otimes v^4_1=T'+a^1_1\otimes c^2_2\otimes (a^3_1+\mu a^3_2)\otimes v^4_2$, where $T'=a^1_1\otimes c^2_1\otimes a^3_1\otimes v^4_1+a^1_1\otimes c^2_1\otimes a^3_2\otimes (\mu v^4_2+u^4_2)+a^1_2\otimes c^2_1\otimes a^3_2\otimes v^4_1\in\widehat{T}_{a^1_1\otimes c^2_1\otimes a^3_2\otimes v^4_1} X_3$.

Subcase 24: If $\lambda a^3_1+\mu a^3_2=a^3_1+\mu a^3_2$, $\mu\neq 0$, $x=1$, $yu^4_1+zu^4_2=u^4_2$, let $c^2_1=a^2_1$, $c^2_2=\mu a^2_2-a^2_1$, $v^4_1=u^4_2+b^4_2$ and $v^4_2=b^4_2$, then $T=(a^1_1+a^1_2)\otimes c^2_1\otimes a^3_2\otimes v^4_1+a^1_1\otimes c^2_2\otimes (\displaystyle\frac{1}{\mu}a^3_1+a^3_2)\otimes v^4_2+a^1_1\otimes c^2_1\otimes a^3_1\otimes (u^4_1+\frac{1-\mu}{\mu}v^4_2)$.

Subcase 25: If $\lambda a^3_1+\mu a^3_2=a^3_1$, $x=1$, $yu^4_1+zu^4_2=u^4_2$, then $T=T'+(a^1_1+a^1_2)\otimes a^2_1\otimes a^3_2\otimes (u^4_2+b^4_2)$, where $T'=a^1_1\otimes a^2_1\otimes a^3_1\otimes u^4_1+a^1_1\otimes(a^2_2-a^2_1)\otimes a^3_1\otimes b^4_2+a^1_1\otimes a^2_1\otimes(a^3_1-a^3_2)\otimes b^4_2\in\widehat{T}_{a^1_1\otimes a^2_1\otimes a^3_1\otimes b^4_2} X_3$.

If $b^4_2$ is in the subspace $V_4$ spanned by $u^4_1$ and $u^4_2$, after a change of basis of $A_3\otimes V_4$ we can assume $b^3_1$ is preserved, and $b^4_2=u^4_1$ or $u^4_2$. So we have:

Subcase 26: If $b^4_2=u^4_1$, $T=a^1_1\otimes (a^2_1+a^2_2)\otimes a^3_1\otimes u^4_1+a^1_1\otimes a^2_1\otimes a^3_2\otimes u^4_2+a^1_2\otimes a^2_1\otimes u^3_3\otimes b^4_3$.

Subcase 27: If $b^4_2=u^4_2$, $T=T'+a^1_2\otimes a^2_1\otimes u^3_3\otimes b^4_3$, where $T'=a^1_1\otimes a^2_1\otimes a^3_1\otimes u^4_1+a^1_1\otimes a^2_1\otimes a^3_2\otimes u^4_2+a^1_1\otimes a^2_2\otimes a^3_1\otimes u^4_2\in\widehat{T}_{a^1_1\otimes a^2_1\otimes a^3_1\otimes u^4_2} X_3$.

Subcase 28: If $b^3_1: A^*_3\rightarrow A_4\otimes\cdots\otimes A_n$ has rank $1$, say $b^3_1=u^3_1\otimes b^4_1$ for some $u^3_1\in A_3$ and $b^4_1\in A_4\otimes\cdots\otimes A_n$, then $T=a^1_1\otimes a^2_1\otimes u^3_1\otimes b^4_1+a^1_1\otimes a^2_2\otimes a^3_1\otimes b^4_2+a^1_2\otimes a^2_1\otimes u^3_3\otimes b^4_3$.


Now we assume $T\in\sigma_3 (X_{k-1})$, and $T$ is of Type 2, but is not of Type 1 when viewed as a tensor in $A_1\otimes A_2\otimes (A_3\otimes\cdots\otimes A_n)$. For each normal form, we show by induction that $T\in \sigma_3 (X)$.

\textbf{Subtype 1}: $T=b^1_1\otimes\cdots\otimes b^k_1+b^1_2\otimes\cdots\otimes b^k_2+b^1_3\otimes\cdots\otimes b^k_3$. Since we assume $T(A^*_3\otimes\cdots\otimes A^*_n)\subset V$, where $V$ is spanned by $a^1_1\otimes a^2_1$, $a^1_1\otimes a^2_2$, and $a^1_2\otimes a^2_1$, then $b^1_j\otimes b^2_j\in A_1\otimes A_2$ has rank $1$ for any $1\leq j\leq 3$ implies $b^1_j=a^1_1$ or $b^2_j=a^2_1$. Hence we have two subcase: $T=a^1_2\otimes a^2_1\otimes b^3_1\otimes\cdots\otimes b^k_1+a^1_1\otimes a^2_2\otimes b^3_2\otimes\cdots\otimes b^k_2+(\lambda a^1_1+\mu a^1_2)\otimes a^2_1\otimes b^3_3\otimes\cdots\otimes b^k_3$ or $T=a^1_1\otimes a^2_2\otimes b^3_1\otimes\cdots\otimes b^k_1+a^1_2\otimes a^2_1\otimes b^3_2\otimes\cdots\otimes b^k_2+a^1_1\otimes (\lambda a^2_1+\mu a^2_2)\otimes b^3_3\otimes\cdots\otimes b^k_3$. Here we only show the first case since the argument for the second case is similar. For the first subcase, if $\lambda=0$, $T$ has been discussed in \textbf{Case 3 Type 1}, so we assume $\lambda\neq 0$. Now let $c^1_1=a^1_2\otimes a^2_1$, $c^1_2=a^1_1\otimes a^2_2$, and $c^1_3=(\lambda a^1_1+\mu a^1_2)\otimes a^2_1$. From the argument of \textbf{Case 2 Type 1}, we can deduce directly that $T\in\sigma_3 (X_k)$ except for the following several subcases.

Exceptional Subcase 1: $b^3_j=a^3_j$ for $1\leq j\leq 2$, $b^3_3=a^3_1+a^3_2$, $b^k_2=a^k_1\otimes u^{k+1}_1+a^k_2\otimes u^{k+1}_2$ for some $u^{k+1}_1$, $u^{k+1}_2\in A_{k+1}\otimes\cdots\otimes A_n$, $b^k_1=b^k_3=a^k_1\otimes u^{k+1}_1$, $b^i_1=b^i_2=b^i_3$ for all $4\leq i\leq k-1$, then there is no harm to assume $k=4$. So $T=(c^1_1+c^1_3)\otimes a^3_1\otimes a^4_1\otimes u^5_2+c^1_2\otimes a^3_2\otimes a^4_1\otimes u^5_1+c^1_2\otimes a^3_2\otimes a^4_2\otimes u^5_2+c^1_3\otimes a^3_2\otimes a^4_1\otimes u^5_2$. When $\mu\neq -1$, $T=[\lambda a^1_1+(\mu+1)a^1_2]\otimes a^2_1\otimes(a^3_1+\displaystyle\frac{\mu}{\mu+1}a^3_2)\otimes a^4_1\otimes u^5_2+T'$, where $T'=a^1_1\otimes a^2_2\otimes a^3_2\otimes a^4_1\otimes u^5_1+a^1_1\otimes a^2_2\otimes a^3_2\otimes a^4_2\otimes u^5_2+a^1_1\otimes \displaystyle\frac{\lambda}{\mu+1}a^2_1\otimes a^3_2\otimes a^4_1\otimes u^5_2\in\widehat{T}_{a^1_1\otimes a^2_2\otimes a^3_2\otimes a^4_1\otimes u^5_2} X_4$. When $\mu=-1$, $T=T'+T''$, where $T'=a^1_1\otimes a^2_1\otimes \lambda a^3_1\otimes a^4_1\otimes u^5_2+(\lambda a^1_1-a^1_2)\otimes a^2_1\otimes a^3_2\otimes a^4_1\otimes u^5_2\in\widehat{T}_{a^1_1\otimes a^2_1\otimes a^3_2\otimes a^4_1\otimes u^5_2} X_4$, and $T''=a^1_1\otimes a^2_2\otimes a^3_2\otimes a^4_1\otimes u^5_1+a^1_1\otimes a^2_2\otimes a^3_2\otimes a^4_2\otimes u^5_2\in\widehat{T}_{a^1_1\otimes a^2_2\otimes a^3_2\otimes a^4_1\otimes u^5_2} X_4$.

Exceptional Subcase 2: $T=(c^1_1+c^1_3)\otimes a^3_1\otimes b^4_1\otimes\cdots\otimes b^{k-1}_1\otimes a^k_1\otimes u^{k+1}_1+(c^1_2+c^1_3)\otimes a^3_2\otimes b^4_1\otimes\cdots\otimes b^{k-1}_1\otimes a^k_1\otimes u^{k+1}_1+c^1_2\otimes a^3_2\otimes b^4_1\otimes\cdots\otimes b^{k-1}_1\otimes a^k_2\otimes u^{k+1}_2$. It is harmless to assume $k=4$. When $\mu\neq -1$, $T=[\lambda a^1_1+(\mu+1)a^1_2]\otimes a^2_1\otimes (a^3_1+\displaystyle\frac{\mu}{\mu+1}a^3_2)\otimes a^4_1\otimes u^5_1+\frac{1}{\mu+1}a^1_1\otimes [(\mu+1)a^2_2+\lambda a^2_1]\otimes a^3_2\otimes a^4_1\otimes u^5_1+a^1_1\otimes a^2_2\otimes a^3_2\otimes a^4_2\otimes u^5_2$. When $\mu=-1$, $T=T'+a^1_1\otimes a^2_2\otimes a^3_2\otimes a^4_2\otimes u^5_2$, where $T'=a^1_1\otimes a^2_1\otimes \lambda a^3_1\otimes a^4_1\otimes u^5_1+a^1_1\otimes a^2_2\otimes a^3_2\otimes a^4_1\otimes u^5_1+(\lambda a^1_1-a^1_2)\otimes a^2_1\otimes a^3_2\otimes a^4_1\otimes u^5_1\in\widehat{T}_{a^1_1\otimes a^2_1\otimes a^3_2\otimes a^4_1\otimes u^5_1} X_4$.

Exceptional Subcase 3: $T=(c^1_1+yc^1_3)\otimes a^3_1\otimes b^4_1\otimes\cdots\otimes b^{k-1}_1\otimes a^k_1\otimes u^{k+1}_1+(c^1_2+yc^1_3)\otimes a^3_2\otimes b^4_1\otimes\cdots\otimes b^{k-1}_1\otimes a^k_1\otimes u^{k+1}_1+c^1_3\otimes (a^3_1+a^3_2)\otimes b^4_1\otimes\cdots\otimes b^{k-1}_1\otimes (xa^k_1+a^k_2)\otimes b^k_3(\alpha^k_2)$ for some $x$, $y\in\mathbb{C}$. It is harmless to assume $k=4$. When $y\mu+1\neq 0$, $T=[y\lambda a^1_1+(y\mu+1)a^1_2]\otimes a^2_1\otimes (a^3_1+\displaystyle\frac{y\mu}{y\mu+1}a^3_2)\otimes a^4_1\otimes u^5_1+a^1_1\otimes (\frac{y\lambda}{y\mu+1}a^2_1+a^2_2)\otimes a^3_2\otimes a^4_1\otimes u^5_1+(\lambda a^1_1+\mu a^1_2)\otimes a^2_1\otimes (a^3_1+a^3_2)\otimes (xa^4_1+a^4_2)\otimes b^4_3(\alpha^4_2)$. When $y\mu+1=0$, $T=T'+(\lambda a^1_1+\mu a^1_2)\otimes a^2_1\otimes (a^3_1+a^3_2)\otimes (xa^4_1+a^4_2)\otimes b^4_3(\alpha^4_2)$, where $T'=a^1_1\otimes a^2_1\otimes y\lambda a^3_1\otimes a^4_1\otimes u^5_1+y(\lambda a^1_1+\mu a^1_2)\otimes a^2_1\otimes a^3_2\otimes a^4_1\otimes u^5_1+a^1_1\otimes a^2_2\otimes a^3_2\otimes a^4_1\otimes u^5_1\in\widehat{T}_{a^1_1\otimes a^2_1\otimes a^3_2\otimes a^4_1\otimes u^5_1} X_4$.


\textbf{Subtype 2}: $T=\displaystyle\sum_{i=1}^k b^1_1\otimes \cdots \otimes b^{i-1}_1\otimes b^i_2\otimes b^{i+1}_1\otimes \cdots \otimes b^k_1+b^1_3\otimes \cdots\otimes b^k_3$. Since $T(A^*_3\otimes\cdots\otimes A^*_n)\subset V$, where $V$ is spanned by $a^1_1\otimes a^2_1$, $a^1_1\otimes a^2_2$ and $a^1_2\otimes a^2_1$, and $b^1_1\otimes b^2_1\in V$ has rank $1$, we can assume $b^1_1=a^1_1$, $b^2_1=a^2_1$. If $b^1_2$ and $b^1_1$ are linearly independent, then assume $b^1_2=a^1_2$, otherwise assume $b^1_3=a^1_2$. If $b^2_2$ and $b^2_1$ are linearly independent, then assume $b^2_2=a^2_2$, otherwise assume $b^2_3=a^2_2$. Since $b^1_3\otimes b^2_3$ is a rank $1$ matrix in $V$, then $b^1_3\otimes b^2_3=(xa^1_1+ya^1_2)\otimes a^2_1$ or $b^1_3\otimes b^2_3=a^1_1\otimes (xa^2_1+ya^2_2)$. Hence we have three subcases:

Subcase 1: $T=\displaystyle\sum_{i=3}^k a^1_1\otimes a^2_1\otimes b^3_1\otimes\cdots\otimes b^{i-1}_1\otimes (b^i_2+\frac{2}{k-2}b^i_1)\otimes b^{i+1}_1\otimes\cdots\otimes b^k_1+a^1_2\otimes a^2_2\otimes b^3_3\otimes\cdots\otimes b^k_3$, which is discussed in \textbf{Case 3 Type 1}.

Subcase 2: $T=(a^1_2\otimes a^2_1+a^1_1\otimes a^2_2)\otimes b^3_1\otimes\cdots\otimes b^k_1+\displaystyle\sum_{i=3}^k a^1_1\otimes a^2_1\otimes b^3_1\otimes\cdots\otimes b^{i-1}_1\otimes b^i_2\otimes b^{i+1}_1\otimes\cdots\otimes b^k_1+a^1_1\otimes a^2_1\otimes b^3_3\otimes\cdots\otimes b^k_3$, which has been discussed in \textbf{Case 3 Type 1} after a change of basis.

Subcase 3: $T=(a^1_2\otimes a^2_1+a^1_1\otimes a^2_2)\otimes b^3_1\otimes\cdots\otimes b^k_3+\displaystyle\sum_{i=3}^k a^1_1\otimes a^2_1\otimes b^3_1\otimes\cdots\otimes b^{i-1}_1\otimes b^i_2\otimes b^{i+1}_1\otimes\cdots\otimes b^k_1+b^1_3\otimes b^2_3\otimes b^3_3\otimes\cdots\otimes b^k_3$, where $b^1_3$ and $a^1_1$ are independent, or $b^2_3$ and $a^2_1$ are independent. Let $c^1_1=a^1_1\otimes a^2_1$, $c^1_2=a^1_2\otimes a^2_1+a^1_1\otimes a^2_2$, $c^1_3=b^1_3\otimes b^2_3$, and $V_1$ denote the subspace of $A_1\otimes A_2$ spanned by $c^1_1$, $c^1_2$ and $c^1_3$, since $b^1_3\otimes b^2_3=(xa^1_1+ya^1_2)\otimes a^2_1$ or $b^1_3\otimes b^2_3=a^1_1\otimes (xa^2_1+ya^2_2)$, by the argument of \textbf{Case 2 Type 2}, we have $T\in \sigma_3 (X_k)$ directly except for a few subcases. From the argument of \textbf{Case 2 Type 2}, we can see it is harmless to assume $k=4$ when considering these exceptional subcases.

Exceptional Case 1: $b^3_j=a^3_j$ for $j=1, 2$, $b^4_j=a^4_1\otimes b^5_1$ for some $b^5_1\in A_5\otimes\cdots\otimes A_n$, $b^3_3=a^3_2+\lambda a^3_1$ for some $\lambda\in\mathbb{C}$, $b^4_3: A^*_4\rightarrow A_5\otimes\cdots\otimes A_n$ has rank $2$, say $b^4_3=a^4_1\otimes u^5_1+a^4_2\otimes u^5_2$, and $b^5_1$ is a linear combination of $u^5_1$ and $u^5_2$, then by redefining $a^1_2$ and $a^2_2$, we can assume $c^1_3=a^1_2\otimes a^2_1$ or $a^1_1\otimes a^2_2$, $u^5_1=b^5_1-\mu u^5_2$ for some $\mu\in\mathbb{C}$ or $u^5_2=b^5_1$. If $c^1_3=a^1_2\otimes a^2_1$, $u^5_1=b^5_1-\mu u^5_2$, then $T=a^1_1\otimes [a^2_2-(1+\lambda)a^2_1]\otimes a^3_1\otimes a^4_1\otimes b^5_1+(a^1_1+a^1_2)\otimes a^2_1\otimes [(1+\lambda)a^3_1+a^3_2]\otimes a^4_1\otimes b^5_1+a^1_2\otimes a^2_1\otimes (a^3_2+\lambda a^3_1)\otimes (a^4_2-\mu a^4_1)\otimes u^5_2$. If $c^1_3=a^1_2\otimes a^2_1$, $u^5_2=b^5_1$, $T=T'+a^1_1\otimes (a^2_2-\lambda a^2_1)\otimes a^3_1\otimes a^4_1\otimes b^5_1$, where $T'=a^1_2\otimes a^2_1\otimes a^3_1\otimes a^4_1\otimes b^5_1+a^1_1\otimes a^2_1\otimes (a^3_2+\lambda a^3_1)\otimes a^4_1\otimes b^5_1+a^1_2\otimes a^2_1\otimes (a^3_2+\lambda a^3_1)\otimes a^4_1\otimes u^5_1+a^1_2\otimes a^2_1\otimes (a^3_2+\lambda a^3_1)\otimes a^4_2\otimes b^5_1\in \widehat{T}_{a^1_2\otimes a^2_1\otimes (a^3_2+\lambda a^3_1)\otimes a^4_1\otimes b^5_1} X_4$. If $c^1_3=a^1_1\otimes a^2_2$, $u^5_1=xb^5_1+yu^5_2$ for some $0\neq x, y\in\mathbb{C}$, $T=a^1_1\otimes (a^2_1+a^2_2)\otimes [xa^3_2+(x\lambda+1)a^3_1]\otimes a^4_1\otimes b^5_1+(a^1_2-\displaystyle\frac{x\lambda+1}{x}a^1_1)\otimes a^2_1\otimes a^3_1\otimes a^4_1\otimes b^5_1+a^1_1\otimes a^2_2\otimes (a^3_2+\lambda a^3_1)\otimes (ya^4_1+a^4_2)\otimes u^5_2$. If $c^1_3=a^1_1\otimes a^2_2$, $u^5_2=b^5_1$, $T=T'+(a^1_2-\lambda a^1_1)\otimes a^2_1\otimes a^3_1\otimes a^4_1\otimes b^5_1$, where $T'=a^1_1\otimes a^2_2\otimes a^3_1\otimes a^4_1\otimes b^5_1+a^1_1\otimes a^2_1\otimes (a^3_2+\lambda a^3_1)\otimes a^4_1\otimes b^5_1+a^1_1\otimes a^2_2\otimes (a^3_2+\lambda a^3_1)\otimes a^4_1\otimes u^5_1+a^1_1\otimes a^2_2\otimes (a^3_2+\lambda a^3_1)\otimes a^4_2\otimes b^5_1\in \widehat{T}_{a^1_1\otimes a^2_2\otimes (a^3_2+\lambda a^3_1)\otimes a^4_1\otimes b^5_1} X_4$. If $x=0$, $b^4_1=b^4_2$ and $b^4_3: A^*_4\rightarrow A_5\otimes\cdots\otimes A_n$ all have rank $1$.

Exceptional Case 2: If $c^1_3=a^1_2\otimes a^2_1$, $b^4_1=b^4_3=a^4_1\otimes b^5_1$ for some $b^5_1\in A_5\otimes\cdots\otimes A_n$, $b^4_2=a^4_1\otimes u^5_1+a^4_2\otimes u^5_2$ for some $u^5_1, u^5_2\in A_5\otimes\cdots\otimes A_n$, and $u^5_1=xu^5_2+yb^5_1$ for some $x, y\in\mathbb{C}$, then $T=a^1_1\otimes [(y-\lambda-1)a^2_1+a^2_2]\otimes a^3_1\otimes a^4_1\otimes b^5_1+(a^1_1+a^1_2)\otimes a^2_1\otimes [a^3_2+(\lambda+1)a^3_1]\otimes a^4_1\otimes b^5_1+a^1_1\otimes a^2_1\otimes a^3_1\otimes [xa^4_1+a^4_2]\otimes u^5_2$. If $c^1_3=a^1_1\otimes a^2_2$, then $T=a^1_1\otimes (a^2_1+a^2_2)\otimes [a^3_2+(\lambda+1)a^3_1]\otimes a^4_1\otimes b^5_1+[a^1_2+(y-\lambda-1)a^1_1]\otimes a^2_1\otimes a^3_1\otimes a^4_1\otimes b^5_1+a^1_1\otimes a^2_1\otimes a^3_1\otimes (xa^4_1+a^4_2)\otimes u^5_2$.


\textbf{Subtype 3}: $T=\displaystyle\sum_{i<j} b^1_1\otimes\cdots\otimes b^{i-1}_1\otimes b^i_2\otimes b^{i+1}_1\otimes \cdots\otimes b^{j-1}_1\otimes b^j_2\otimes b^{j+1}_1\otimes\cdots\otimes b^k_1+\displaystyle\sum_{i=1}^k b^1_1\otimes\cdots\otimes b^{i-1}_1\otimes b^i_3\otimes b^{i+1}_1\otimes \cdots\otimes b^k_1$. If $b^1_2=b^1_1$, $b^2_2=b^2_1$ up to a scalar, then we can assume $b^1_2=b^1_1=a^1_1$, $b^2_2=b^2_1=a^2_2$, $b^1_3=a^1_2$, and $b^2_3=a^2_1$. This has been discussed in \textbf{Case 3 Type 1}. Otherwise, Let $c^1_1=b^1_1\otimes b^2_1$, $c^1_2=b^1_2\otimes b^2_1+b^1_1\otimes b^2_2$, and $c^1_3=b^1_2\otimes b^2_2+b^1_3\otimes b^2_1+b^1_1\otimes b^2_3$. By the argument of \textbf{Case 2 Type 3}, we can see $T\in\sigma_3 (X_k)$ except only one subcase, and by the argument of \textbf{Case 2 Type 3} it is harmless to assume $k=4$ for the exceptional subcase, $b^3_j=a^3_j$ for $j=1, 2$, $b^3_3=xa^3_1+ya^3_2$ for some $x, y\in\mathbb{C}$, $b^4_1=b^4_2=a^4_1\otimes b^5_1$ for some $b^5_1\in A_5\otimes \cdots\otimes A_n$, $b^4_3: A^*_4\rightarrow A_5\otimes \cdots\otimes A_n$ has rank $2$, say $b^4_3=a^4_1\otimes u^5_1+a^4_2\otimes u^5_2$ for some $u^5_1, u^5_2\in A_5\otimes \cdots\otimes A_n$, $u^5_2$ and $b^5_1$ are linearly independent, and $u^5_1=\lambda u^5_2+\mu b^5_1$ for some $\lambda, \mu\in\mathbb{C}$. So $T=[(x+\mu-5)c^1_1+2c^1_2+c^1_3]\otimes a^3_1\otimes a^4_1\otimes b^5_1+[(y+2)c^1_1+c^1_2]\otimes a^3_2\otimes a^4_1\otimes b^5_1+c^1_1\otimes a^3_1\otimes (\lambda a^4_1+a^4_2)\otimes u^5_2$, which has been discussed in \textbf{Case 3 Type 1}.


\textbf{Subtype 4}: $T=\displaystyle\sum_{i=2}^k b^1_2\otimes b^2_1\otimes\cdots\otimes b^{i-1}_1\otimes b^i_2\otimes b^{i+1}_1\otimes\cdots\otimes b^k_1+\sum_{i=1}^k b^1_1\otimes\cdots\otimes b^{i-1}_1\otimes b^i_3\otimes b^{i+1}_1\otimes\cdots\otimes b^k_1$. If $b^1_2=b^1_1$ up to a scalar, $T$ has been discussed in \textbf{Case 3 Type 1}. Otherwise, let $c^1_1=b^1_1\otimes b^2_1$, $c^1_2=b^1_2\otimes b^2_1$, $c^1_3=b^1_2\otimes b^2_2+b^1_3\otimes b^2_1$. From the argument of \textbf{Case 2 Type 4}, we can see $T\in\sigma_3 (X_k)$.

\subsection{Case 4: $T\in\sigma_2(X_2)$}

We assume $T \in \sigma_2(X_{k-1})$, and show $T\in\sigma_3 (X_k)$ by checking each type of the normal forms in Proposition~\ref{prop:norm}.

\textbf{Type 1}: $T=b^1_1\otimes\cdots\otimes b^k_1$. Then $T=b^1_1\otimes\cdots\otimes b^{k-1}_1\otimes a^k_1\otimes b^k_1(\alpha^k_1)+b^1_1\otimes\cdots\otimes b^{k-1}_1\otimes a^k_2\otimes b^k_1(\alpha^k_2)+b^1_1\otimes\cdots\otimes b^{k-1}_1\otimes a^k_3\otimes b^k_1(\alpha^k_3)$.

\textbf{Type 2}: $T=b^1_1\otimes \cdots\otimes b^k_1+b^1_2\otimes \cdots\otimes b^k_2$. Since there is some $1\leq i\leq k-1$ such that $b^i_1$ and $b^i_2$ are linearly independent, then $\dim T(A^*_i\otimes A^*_k)\leq 3$ implies at least one of $b^k_1$ and $b^k_2: A^*_k\rightarrow A_{k+1}\otimes\cdots\otimes A_n$ has rank $1$, and the other one has rank at most $2$, say $b^k_1=a^k_1\otimes b^{k+1}_1$ and $b^k_2=a^k_1\otimes b^{k+1}_2+a^k_2\otimes b^{k+1}_3$ for some $b^{k+1}_1$, $b^{k+1}_2$, $b^{k+1}_3\in A_{k+1}\otimes\cdots\otimes A_n$. Hence, $T=b^1_1\otimes \cdots\otimes b^{k-1}_1\otimes a^k_1\otimes b^{k+1}_1+b^1_2\otimes \cdots\otimes b^{k-1}_2\otimes a^k_1\otimes b^{k+1}_2+b^1_2\otimes\cdots\otimes b^{k-1}_2\otimes a^k_2\otimes b^{k+1}_3$.

\textbf{Type 3}: $T=\displaystyle\sum_{i=1}^k b^1_1\otimes \cdots\otimes b^{i-1}_1\otimes b^i_2\otimes b^{k+1}_1\otimes\cdots\otimes b^k_1$. Without loss of generality, we can assume $b^1_1$ and $b^1_2$ are linearly independent, and $b^2_1$ and $b^2_2$ are linearly independent, then $\dim T(A^*_1\otimes A^*_k)\leq 3$ implies $b^k_1: A^*_k\rightarrow A_{k+1}\otimes \cdots\otimes A_n$ has rank $1$, say $b^k_1=a^k_1\otimes b^{k+1}_1$ for some $b^{k+1}_1\in A_{k+1}\otimes\cdots\otimes A_n$, and $\{b^{k+1}_1$, $b^k_2(\alpha^k_2)$, $b^k_2(\alpha^k_3)\}$ spans an at most $2$ dimensional subspace. Thus we can assume $b^k_2(\alpha^k_3)=xb^{k+1}_1+yb^k_2(\alpha^k_2)$ for some $x$, $y\in\mathbb{C}$, then $T=T'+a^1_1\otimes \cdots\otimes b^{k-1}_1\otimes (a^k_2+ya^k_3)\otimes b^k_2(\alpha^k_2)$, where $T'=\displaystyle\sum_{i=1}^{k-2} b^1_1\otimes \cdots\otimes b^{i-1}_1\otimes b^i_2\otimes b^{i+1}_1\otimes \cdots\otimes b^{k-1}_1\otimes a^k_1\otimes b^{k+1}_1+b^1_1\otimes \cdots\otimes b^{k-1}_1\otimes a^k_1\otimes b^k_2(\alpha^k_1)+a^1_1\otimes\cdots\otimes b^{k-1}_1\otimes xa^k_3\otimes b^{k+1}_1\in\widehat{T}_{b^1_1\otimes \cdots\otimes b^{k-1}_1\otimes a^k_1\otimes b^{k+1}_1} X_k$.

\section{Acknowledgement}

I wish to thank J. M. Landsberg for bringing the question to my attention and for many fruitful discussions during the course
of this work. I am grateful to Seth Sullivant and Luke Oeding for helpful conversations.

\bibliographystyle{amsplain}

\end{document}